\documentclass[10pt]{amsart}
\usepackage{amscd}
\usepackage{amssymb}
\usepackage{amsthm}
\usepackage{enumerate}
\usepackage[normalem]{ulem}
\usepackage{mathtools}
\usepackage[usenames,dvipsnames]{xcolor}
\usepackage{stmaryrd}
\usepackage[mathscr]{eucal}
 \usepackage{cite}
\usepackage{upgreek}
\usepackage[bookmarks=false]{hyperref}
\usepackage[T2A]{fontenc}
\usepackage[utf8]{inputenc}
\usepackage{tikz-cd}
\usepackage{mathrsfs}

\newtheorem{introthm}{Theorem}

\newtheorem{thm}{Theorem}[section]
\newtheorem{lem}[thm]{Lemma}
\newtheorem{prop}[thm]{Proposition}
\newtheorem{cor}[thm]{Corollary}

\theoremstyle{definition}
\newtheorem{defn}[thm]{Definition}

\newtheorem{ex}{Example}

\newcommand\kh{\operatorname{kh}}
\newcommand{\stacked}[3]{\underset{#3}{\overset{#2}{#1}}}

\theoremstyle{remark}
\newtheorem{rem}[thm]{Remark}

\numberwithin{equation}{section}

\newcommand{\bC}{{\mathbb C}}

\newcommand{\bM}{{\mathbb M}}
\newcommand{\bN}{{\mathbb N}}

\newcommand{\bR}{{\mathbb R}}

\newcommand{\bT}{{\mathbb T}}

\newcommand{\bZ}{{\mathbb Z}}

\newcommand{\cA}{{\mathcal A}}
\newcommand{\cB}{{\mathcal B}}

\newcommand{\cH}{{\mathcal H}}

\newcommand{\cK}{{\mathcal K}}

\newcommand{\cM}{{\mathcal M}}

\newcommand{\cU}{{\mathcal U}}

\DeclareMathOperator{\id}{id}

\DeclareMathOperator{\tr}{tr}

\DeclareMathOperator{\Span}{span}

\DeclareMathOperator{\rank}{rank}

\DeclareMathOperator{\HS}{HS}

\newcommand{\ip}[1]{\langle #1 \rangle}

\begin{document}

\title{Selfless reduced free product $C^*$-algebras}

\author[B. Hayes]{Ben Hayes}
\address{\parbox{\linewidth}{Department of Mathematics, University of Virginia\\
141 Cabell Drive, Kerchof Hall
P.O. Box 400137,
Charlottesville, VA 22904}}
\email{brh5c@virginia.edu}

\author[S. Kunnawalkam Elayavalli]{Srivatsav Kunnawalkam Elayavalli}
\address{\parbox{\linewidth}{Department of Mathematics, University of California, San Diego, \\
9500 Gilman Drive \# 0112, La Jolla, CA 92093}}
\email{skunnawalkamelayaval@ucsd.edu}
\urladdr{https://sites.google.com/view/srivatsavke}

\author[L. Robert]{Leonel Robert}
\address{\parbox{\linewidth}{Department of Mathematics, University of Louisiana at Lafayette, \\
217 Maxim Doucet Hall, 1401 Johnston Street, Lafayette, LA 70503, USA}}
\email{lrobert@louisiana.edu}

\begin{abstract}
   We study selflessness in the general setting of reduced free products of $C^*$-algebras. Towards this end, we develop a suitable theory of rapid decay for filtrations in arbitrary $C^*$-probability spaces. We provide several natural examples and permanence properties of this phenomenon. By using this framework in combination with von Neumann algebraic techniques involving  approximate forms of orthogonality, we are able to prove selflessness for general families of reduced free product $C^*$-algebras. As an instance of our results, we prove selflessness and thus strict comparison for the canonical $C^*$-algebras generated by Voiculescu's free semicircular systems. Our results also provide new examples of purely infinite reduced free products.
   
\end{abstract}

\maketitle%

\section{Introduction}

\subsection{Context and motivation}

Since its introduction by Blackadar in \cite{Blackadar_1989}, the  \emph{strict comparison} property has played a central role in the modern theory of the structure and classification of $C^*$-algebras.
This property can be interpreted as an algebraic order-completeness of the Cuntz semigroup, also known as almost unperforation \cite{Rordam}. When present, strict comparison yields insights into the structure of both the tracial simplex and the $K$-theoretic invariants of a $C^*$-algebra, which are essential to Elliott’s classification program for simple nuclear
$C^*$-algebras, as demonstrated in \cite{GonglinNiu1, GongLinNiu2, uniformgamma, carrión2023classifyinghomomorphismsiunital, winter2017structurenuclearcalgebrasquasidiagonality, white2023abstractclassificationtheoremsamenable}. The influential Toms-Winter conjecture for nuclear $C^*$-algebras further highlights the significance of strict comparison, relating it to other key regularity properties such as $\mathcal{Z}$-stability and nuclear dimension (see, for example, the important work \cite{MatuiSato}).


Studying strict comparison in $C^*$-algebras has by now several demonstrably important applications. 
To illustrate this, let us fix, for convenience,  a simple monotracial $C^*$-algebra $A$ with  strict comparison. We can then explicitly calculate 
the Cuntz semigroup of $A$ (e.g., see \cite{Tensorproductsthiel}). We know also, by a recent theorem of Lin,  that $A$ must have stable rank one \cite{lin2024strictcomparisonstablerank}. Further, by the work \cite{robert2012classification}, we can  fully classify  the  embeddings of the Jiang-Su algebra $\mathcal{Z}$ into $A$ up to approximate unitary equivalence (note also, the upcoming work \cite{unital2} for related results).  More recently, strict comparison led in \cite{tarski} to solve the $C^*$-algebraic Tarski problem, strikingly in the negative. In particular, it was proved there that the $K_1$-group of an ultraproduct of $C^*$-algebras can be computed under a strict comparison hypothesis. 


Despite its natural formulation (analogous to the foundational Murray-von Neumann comparison \cite{Murray-Neumann36}), establishing strict comparison has been challenging, especially outside  the class of nuclear $C^*$-algebras.
By a theorem of R{\o}rdam, $C^*$-algebras that tensorially absorb the Jiang-Su algebra have strict comparison \cite{Rordam}. However, key examples such as $C_\lambda^*(\mathbb F_2)$ fail to be Jiang-Su stable (and are, in fact,  tensorially prime).
Dykema and R{\o}rdam proved strict comparison  for certain infinite free products by using a subtle free probability argument \cite{dykema1998projections, dykema2}. Apart from this and a few other ad-hoc cases (for instance all II$_1$ factors viewed as $C^*$-algebras have strict comparison) the problem of proving strict comparison for naturally occurring examples of non-nuclear $C^*$-algebras, such as $C_\lambda^*(\mathbb F_2)$, remained open for many years. 


In the  recent work \cite{robert2023selfless}, the third-named author introduced the notion of a \emph{selfless} $C^*$-probability space. Selflessness yields several structural properties in a unified manner: simplicity, at most one tracial state, if the state is faithful the 
dichotomy ``purely infinite or stable rank one'', and, significantly, strict comparison. An important breakthrough came  then in \cite{sri2025strictcomparisonreducedgroup}, where it was proved that large classes of reduced group $C^*$-algebras are selfless, including those arising from acylindrically hyperbolic groups with the rapid decay property and trivial finite radical. This settled affirmatively the question of strict comparison for $C_\lambda^*(\mathbb F_2)$.
The strategy developed in \cite{sri2025strictcomparisonreducedgroup}  (notably inspired by the key work \cite{louder2022strongly}) has led to a more recent advancement that obtains selflessness and, therefore, strict comparison and stable rank one, for first examples of reduced group $C^*$-algebras of higher-rank lattices  \cite{vigdorovich2025structuralpropertiesreducedcalgebras} 
(see also \cite{gelanderavni}).

Despite these recent developments, the question of strict comparison remains unexplored for natural families of $C^*$-algebras outside of the reduced group $C^*$-algebra class. An important and rather canonical example here is Voiculescu's free semicircular $C^*$-algebras $\mathcal{S}_n$, emerging from his revolutionary free probability theory \cite{weakconvergence, freerandomvariables} (see also \cite{HTGod, Haageruptraces}).  The free semicircular $C^*$-algebras appear to be a more natural $C^*$-algebraic analogue of the free group factors than the reduced free group $C^*$-algebras, especially for $K$-theoretic reasons \cite{projectionsPicu, KgroupsPicu}. Note that they share the same Elliott invariant as the Jiang-Su algebra $\mathcal{Z}$. These $C^*$-algebras are also of interest because they are a canonical way of constructing a tracial $C^*$-algebra from a real Hilbert space, as Voiculescu shows in \cite{voiculescufreesemi}. For instance, every second countable locally compact group admits a canonical action on a semicircular $C^*$-algebra by applying Voiculescu's functor to the left regular representation on $L^2(G)$ (as a real Hilbert space). Intriguingly, the basic question of whether these $C^*$-algebras are pairwise isomorphic for $n\neq m$ remains open. Notably even the question of strict comparison for these families has remained an open problem. We thank N. Chris Phillips and J. Gabe for communicating this to us. 

\subsection{Main results}
In this article 
we are able to  push the envelope from \cite{sri2025strictcomparisonreducedgroup}, in particular moving away from the class of reduced group $C^*$-algebras. 
We obtain the following theorem, answering affirmatively the open question described above. 

\begin{introthm}\label{main semi thm}
    Voiculescu's free semicircular $C^*$-algebras $\mathcal S_n$ are selfless (in the sense of \cite{robert2023selfless}) for $n\geq 2$. In particular, they have strict comparison. 
\end{introthm}


Theorem \ref{main semi thm} will be deduced as a special case of our general results later in the introduction (Theorem \ref{mainthm3}). In order to motivate our upcoming results we must first discuss the notion of \emph{rapid decay}. This is a delicate inequality bounding the operator norm by a polynomial distortion of the $2$-norm, for finitely supported elements of the group algebra. It was discovered by Haagerup in his  seminal work \cite{HaagerupGod}, and has since occupied a major place in analytic group theory (see also \cite{JoliRDP, Jolissaint, delaharpe}). Rapid decay has played a crucial role in important problems such as the Baum-Connes and Novikov conjectures \cite{LafforgueBaumConnes}, deep questions in probabilistic group theory \cite{Grigorchuk}, and of course also in the recent work \cite{sri2025strictcomparisonreducedgroup}. In the present work we develop a framework for rapid decay in the general context of \emph{filtrations} of $C^*$-algebras.  By a filtration of a unital $C^*$-algebra
$A$ we understand an increasing sequence of selfadjoint vector subspaces $(V_n)_{n=0}^\infty$
with $V_0=\bC\cdot 1$, $V_mV_{n}\subseteq V_{n+m}$ for all $n,m$, and with dense union in $A$.
A basic reason to work with filtrations is that they are a natural analogue of group-generating sets, allowing us to conveniently implement combinatorial arguments later on. We then formulate the rapid decay property of a $C^*$-probability space $(A,\rho)$ in reference to a filtration of $A$. This framework is crucial for the applications in this paper, and we expect it to be of independent interest.
We develop it in Sections \ref{sec: rapid decay} and \ref{RD free product section}. Note that rapid decay has been considered in the literature in analogous contexts outside the setting of groups, see for instance \cite{TakaRieffel, VergnRD, Brannan_2014, Miyagawa}.

We show that several natural families of $C^*$-algebras admit filtrations with rapid decay. This is discussed in detail in Section \ref{RD filtrations examples}. Interestingly, in several examples, proving rapid decay for filtrations is closely tied to the theory of orthogonal polynomials \cite{SzegoOrthoPolys}.  Here we  list some key examples:
\begin{itemize}
\item $C^*_\lambda(G)$ with the canonical trace, where $G$ is any group with rapid decay in the traditional sense, admitting a finite generating set $S$ with $S=S^{-1}$. See the surveys \cite{sapir2015rapid, ChatterjiRDP} for a variety of examples.
\item The semicircular $C^*$-algebra $A=C([-2,2])$, with $\tau$ being integration against the semicircular distribution, i.e. $\tau(f)=\frac{1}{2\pi}\int_{-2}^{2}f(t)\sqrt{4-t^{2}}\,dt$. More generally, $(C([a,b]),\int\cdot \, d\mu)$ for any atomless probability measure $\mu$ with support $[a,b]$.
    \item $A$ any separable AF algebra and $\rho$ a faithful state (not necessarily assumed to be a trace) on $A$.
    \item $C(G)$ where $G$ is a connected compact Lie group and letting $\tau$ be the state on $C(G)$ induced by the Haar measure.
    \item Let $G$	be a connected compact Lie group and let $H$ be a closed subgroup. Identify $C(G/H)$ with the $C^*$-subalgebra of $C(G)$ of 
right $H$-invariant functions, i.e., such that $f(gh)=f(g)$
for all $g\in G$ and $h\in H$. Then $(C(G/H),\tau|_{C(G/H)})$
has a filtration with rapid decay.
\end{itemize}

Rapid decay is preserved under  certain natural operations such as direct sums, corners, certain tensor products (see Proposition \ref{RDpreservation}), and more crucially for our applications,  reduced free products. To prove the preservation of rapid decay under reduced free products we rely on the 
Khintchine type  inequality of Ricard and Xu \cite{KhinRX}. 
This is done in Section \ref{RD free product section}.



Now recall that the roadmap for proving  selflessness is to produce a diagonally freely independent Haar unitary in the $C^*$-ultrapower. In the group setting, \cite{sri2025strictcomparisonreducedgroup} used a two-pronged approach of combining rapid decay with an effective version of the mixed identity freeness property (see for instance \cite{MIF}). Roughly speaking, this local freeness property is an algebraic condition that enables the choice of a candidate for a $C^*$-free Haar unitary in the $C^{*}$-ultrapower. Effectivizing the growth of lifts of this element and combining it with rapid decay is what allows one to penetrate into the entire $C^*$-algebra via maps that are $L^2$-$L^2$ isometries. In the pursuit of adapting this strategy and pushing it outside of the setting of groups, we prove the following result:

\begin{introthm}\label{mainthm2}
Let $(A_{j},\rho_{j})$, for $j=1,2$,  be $C^{*}$-probability spaces with filtrations $(V_{n,j})_{n=0}^{\infty}$ with the rapid decay property. Suppose there exist unitaries $u\in \bigcup_{n}V_{n,1}$, and $v,w\in \bigcup_{n}V_{n,2}$, such that $v$ lies in  $A_{2}^{\rho_{2}}$ (the centralizer of $\rho_2$), and 
\[
\rho_1(u)=\rho_2(v)=\rho_2(w)=\rho_{2}(v^{*}w)=0. 
\]
Then $(A_{1},\rho_{1})*(A_{2},\rho_{2})$ is selfless.
\end{introthm}

The conditions on the unitaries above are motivated by, and closely related to, work of Avitzour \cite{avitzour1982free}.
Note that a finite-dimensional $C^{*}$-probability space $(A,\rho)$ with faithful state $\rho$ always has the rapid decay property relative to the filtration $V_{n}=A$ for all $n$.  
We thus deduce from Theorem \ref{mainthm2} that 
\[
(\mathbb{M}_m(\mathbb{C}),\rho_1)*(\mathbb{M}_n(\mathbb{C}),\mathrm{tr})
\]
is selfless whenever $m,n\geq 2$ and $\rho_{1}$ is faithful.
If $\rho_1$ is additionally nontracial, then 
this reduced free product is purely infinite. 
We note that Theorem \ref{mainthm2} (and Theorem \ref{mainthm3} below), provide relatively simple criteria for proving that a reduced free product is purely infinite (by establishing selflesness relative to a faithful nontracial state). 
Previous results (for example, \cite{dykemapiII})  established pure infiniteness of reduced free products  in certain cases, under more complicated hypotheses.


Other examples of selfless reduced free products can be obtained from Theorem \ref{mainthm2}, e.g.,  $(A,\tau)*(\bM_{n}(\bC),\tr)$ if $n\geq 2$, $A$ is finite-dimensional, and $\tau$ is a faithful trace with $\tau(z)<1/2$ for every central projection in $A$. 



The assumption in Theorem \ref{mainthm2} that the filtrations of $A_1$ and $A_2$ contain unitaries with state zero is restrictive, since filtrations may even fail to have any nonconstant unitaries. Indeed, note that the filtration by polynomials on $C([a,b])$ does not have nonconstant unitaries. In particular, one \emph{cannot} derive Theorem \ref{main semi thm} from Theorem \ref{mainthm2} and therefore this application necessitates new ideas. In our upcoming general result, we are able to significantly expand and abstract the strategy behind Theorem \ref{mainthm2}. We still assume that $A_{1},A_2$ have filtrations with the rapid decay property. But our additional assumptions apply only to the structure of the GNS completion of $A_{2}$ with respect to $\tau$:

\begin{introthm}\label{mainthm3}
    Let $(A_{1},\rho),(A_{2},\tau)$ be unital $C^{*}$-probability spaces with $\tau$ tracial, 
$\dim_{\bC}(A_{1})>1$, 
and $A_{2}$ separable. Suppose that $A_{1},A_{2}$ have filtrations with the rapid decay property. Suppose that either:
\begin{itemize}
    \item the GNS completion of $A_{2}$ with respect to $\tau$ is a $\textrm{II}_{1}$-factor, or
    \item the GNS completion of $A_{2}$ with respect to $\tau$ has diffuse central sequence algebra (e.g. if the GNS completion of $A_{2}$ has diffuse center). 
\end{itemize}
Then $A_{1}*A_{2}$ is selfless.
\end{introthm}
The power of this theorem stems from the fact that very mild restrictions are imposed on $A_1$,  while the two bullet point conditions  on $A_2$ are not particularly restrictive.
We show that having a filtration with the rapid decay property implies faithful GNS representation (see Proposition \ref{RDfaithfulGNS}), and so the first bullet point is equivalent to  $\dim(A_{2})=+\infty$, and  $\tau$ being an extreme point of the tracial Choquet simplex of $A_2$. This is the case for example
when $A_2=C^*_\lambda(G)$ and $G$ is an icc (infinite conjugacy class) group.

The second bullet point holds whenever the GNS completion  of $A_{2}$ with respect to $\tau$ has diffuse center. Thus, we may apply Theorem \ref{mainthm3} when $A_{2}=C(X)$ and $\tau=\int \cdot \,d\mu$ where $\mu$ is atomless, provided that we have filtrations of $A_{1},A_{2}$ with the rapid decay property.  This 
allows us to deduce that the $C^{*}$-algebra of $n$ free semicirculars is selfless (Theorem \ref{main semi thm}) once $n\geq 2$.
As explained before, this case is not directly approachable via Theorem \ref{mainthm2}, since our filtrations for the semicircular measure do not have nonconstant unitaries. We then deduce from Theorem \ref{main semi thm}, that the free Araki--Woods C*-algebras $\Gamma(H_\bR,U_t)$ introduced by Shylakhtenko
 are selfless, simple, and purely infinite, for $\dim H_\bR\geq 3$.

 Yet another case when the second bullet point assumption in Theorem \ref{mainthm3} holds is when $A_2$ is a tensor product $B\otimes C$, where $C$ is unital, simple, separable,  nuclear, and  with unique trace. This is the case, for example, when $A_{2}$ is $\mathcal{Z}$-stable.

Since we do not require that $A_{1}$ has unitaries with state zero, we can use  Theorem \ref{mainthm3} to show that 
\[
(\bC^{2},\rho)*(C^{*}_{\lambda}(G),\tau)
\]
is selfless, when $\rho$ is any faithful state, $G$ is any icc group with the rapid decay property, and $\tau$ is the usual trace from the left regular representation. This does not follow from Theorem \ref{mainthm2}.
Note however that Theorem \ref{mainthm2} does imply that free products such as  
\[
(\bM_{k}(\bC),\mathrm{tr})*(C^{*}_\lambda(G),\tau),\]  with $k\geq 2$ and $G$ a finite nontrivial group, are selfless. These cases do not directly follow Theorem \ref{mainthm3}, as the assumptions of that theorem imply that one of the sides of the free product has diffuse GNS completion. The two theorems are thus complementary in some sense.


Using both cases of Theorem \ref{mainthm3},
we completely classify selflessness among free products of finite dimensional abelian $C^*$-algebras: 

\begin{introthm}\label{abelianfindimmaintheorem}
Let	$A$ and $B$ be finite-dimensional abelian $\mathrm{C}^*$-algebras endowed with faithful states $\tau_A$ and $\tau_B$. Then
$(A,\tau_A)\ast(B,\tau_B)$ 
is selfless if and only if  $\dim(A)+\dim(B)\ge 5$ and  
whenever $p$ is a minimal projection in $A$, and $q$ a minimal projection in $B$, we have $\tau_A(p) + \tau_B(q) < 1$.
\end{introthm}

Besides Theorem \ref{mainthm3},
to prove Theorem \ref{abelianfindimmaintheorem} we rely on Dykema's structure result on  the reduced free products of
finite-dimensional abelian C*-algebras from \cite{Dykemasimplicitystablerankonefree}, as well as 
on Dykema's classification of von Neumann free products of finite-dimensional abelian von Neumann algebras
\cite{dykemafreehyper}.
    We mention that if $(M_{1},\rho_1)$ and $(M_{2},\rho_2)$  are von Neumann algebras equipped with faithful, normal, states, then \cite[Theorem 3]{DykemaFreeProdvNA} determines precisely when their von Neumann free product is a type \textrm{III} factor (see \cite[Theorem 3]{UedaFreeProdVNA} for an alternative approach).
    Furthermore, a rather surprising result of Hartglass-Nelson \cite{HNFAW} precisely classifies the free products of finite-dimensional von Neumann algebras in terms of the Free Araki--Woods factors of Shlyakhtenko \cite{Shlyakhtenko}. It is plausible that these results  can be used to investigate selflessness of the reduced free product of two finite-dimensional $C^{*}$-probability spaces. However, adapting our approach 
   to this setting poses substantial challenges. We leave this question as an interesting problem for future research. 
   

\subsection{Roadmap of the proofs}

We encourage the reader to consult this subsection in parallel with the proofs in Sections \ref{Avitzour section}, \ref{sec: here is where the fun begins} and \ref{fun continues}. 

The proof of Theorem \ref{mainthm2} is covered in Section \ref{Avitzour section}. We recommend the reader to read through this proof first, as it provides a good starter as far as the intuition for the later parts of the paper is concerned. The strategy is heavily inspired by the overall approach of \cite{sri2025strictcomparisonreducedgroup}, and also the work of Avitzour in \cite{avitzour1982free}. The availability of state zero unitaries \emph{inside} the filtrations allows for an elegant transport of the  ideas of \cite{sri2025strictcomparisonreducedgroup}, where the rapid decay property is combined with $L^2$-$L^2$ isometries induced by a natural sequence of *-homomorphisms (similar to Proposition 3.1 of \cite{sri2025strictcomparisonreducedgroup}). 

Let us briefly expand on the arguments. 
Let $(A_{j},\rho_{j}),j=1,2$ be $C^{*}$-probability spaces. For a unitary $v\in \mathcal U(A_1*A_2)$, and letting
$B_{1}=A_{1}$, $B_{2}=A_{2}$, $B_3=A_1$, denote by 
\[
\phi_{v}\colon B_{1}*_{\textnormal{alg}}B_{2}*_{\textnormal{alg}}B_{3}\to A_{1}*A_{2}
\]
the unique *-homomorphism that is the identity on
$B_1,B_2$ and conjugation by $v$ on $B_3$.
Now  let us assume the \emph{Avitzour condition}, i.e., that there exist $u\in \cU(A_{1})$ and $v,w\in \cU(A_{2})$  with $v\in A_{2}^{\rho_{2}}$ and $\rho_{2}(v^{*}w)=0=\rho_{2}(v)=\rho_{2}(w)=\rho_{1}(u)$. For each $n\in \bN$, let
$\phi_n = \phi_{x_n^*}$, 
where $x_{n}=(wuw)(uv)^{n}$. The Avitzour condition allows us to show  that $\phi_n$ is an  $L^2$-$L^2$ isometry on the subspace  of $B_{1}*_{\textnormal{alg}}B_{2}*_{\textnormal{alg}}B_{3}$ spanned by alternating centered words of length less than $n$ (see Lemmas \ref{lem: Avitzour argument} and \ref{item: asy L2 L2 isometry avitzour}). If we assume further that $u,v,w$ belong to  filtrations 
of $A_1$ and $A_2$ with the rapid decay property, then we can make use of the rapid decay property in $A_1*A_2$ to show that the maps $\phi_n$ are asymptotically contractive. This in turn yields the desired *-homomorphism from $B_1*B_2*B_3$ into $(A_1*A_2)^{\omega}$ encoded in the selfless  property.

We now discuss Theorem \ref{mainthm3}. The method of proof here again involves maps $\phi_n$ of the form 
$\phi_{v_n}$, arising from unitaries $v_n$, as above. In the process of adapting the proof strategy for Theorem \ref{mainthm2} to more general free products, we no longer wish to choose these unitaries inside the filtrations of $A_1$ and $A_2$. Instead, we will choose  $v_n\in \mathcal U(A_2)$. In order to get the $L^2$-$L^2$ isometry argument to go  through,
we choose  $(v_n)_n$ satisfying suitable asymptotic orthogonality conditions. This is accomplished in Lemma \ref{lem: asy L2-L2 isometry}. Note that the 
images under $\phi_n$ of the 
subspaces from the filtration of $B_{1}*_{\textnormal{alg}}B_{2}*_{\textnormal{alg}}B_{3}$ 
are no longer contained in subspaces from the filtration of $A_1*A_2$.
Thus, the new arguments need to ensure that these images are contained in subspaces of $A_1*A_2$ where the operator norm is still controlled by the $L^{2}$-norm, i.e., a form of ``inflated rapid decay''. This step is more subtle. We discuss next in more detail the challenges that it entails.  

\subsection*{The rapid decay side} 
Let $E\subseteq A_{2}$ be a vector subspace of the filtration of $A_2$. Suppose that $v\in \mathcal U(A_2)$ is such
that the spaces $vE$ and $E$ are almost orthogonal. By combinatorial considerations (Lemma \ref{lem: combinatorial strucutre of conjugation map}), we are led to consider 
\[F=(\id-\tau)(Ev+v^{*}E+\widehat{F}),\]
where $\widehat{F}$ almost contains (in operator norm) $E$ and $v^{*}Ev$.
By assumption, we have some $C>0$ satisfying $\|x\|\leq C\|x\|_{2}$ for all $x\in E$.  
In order  to maintain good enough rapid decay-style estimates to deduce selflessness we would like to show that we can choose $v$ to additionally satisfy
\[\|y\|\leq \alpha C\|y\|_{2}\]
for some uniform $\alpha>0$ and all $y\in F$.  

Given $x_{1},x_{2}\in E$, $x_{3}\in \widehat{F}$ we have that
\[\|x_{1}v+v^{*}x_{2}+x_{3}\|\leq \|x_{1}\|+\|x_{2}\|+\|x_{3}\|.\]
The first two terms on the right-hand side can be estimated by $C(\|x_{1}\|_{2}+\|x_{2}\|)$. Moreover, from the assumption of almost orthogonality between 
$vE$ and $E$  we see that $\|x_{1}v+v^{*}x_{2}+x_{3}\|_{2}^{2}$ is close to
\[\|x_{1}\|_{2}^{2}+\|x_{2}\|_{2}^{2}+\|x_{3}\|_{2}^{2}.\]
We will address comparing the operator and $L^{2}$-norms of $x_{3}$ later, as we  can only handle this by splitting into cases where the GNS completion either has diffuse central sequence algebra or is a $\textrm{II}_{1}$-factor. For now, let us simply assume that we have appropriate control on the comparison of these norms. 

Combining the above, we see that 
\[\|x_{1}\|+\|x_{2}\|+\|x_{3}\|\leq  C'\sum_{j=1}^{3}\|x_{j}\|_{2}^{2},\]
and the right-hand side above is approximately $C''\|x_{1}v+v^{*}x_{2}+x_{3}\|_{2}^{2}$. 

\subsection*{The general technical approach} We obtain a general blueprint theorem precisely as follows. This is stated in quite technical terms, so we recommend the reader to consult this subsection when reading in parallel with Sections \ref{sec: here is where the fun begins}, \ref{fun continues}.

\begin{introthm}[Theorem \ref{thm: selfless from asy angle conditions}]\label{thmE}
Let $(A_{1},\rho),(A_{2},\tau)$ be unital $C^{*}$-probability spaces with $\tau$ tracial. Suppose that $A_{1}$ and $A_2$ have the rapid decay property
relative to filtrations $(V_{n,1})_{n=0}^{\infty}$ and $(V_{n,2})_{n=0}^{\infty}$.

Suppose that for each $n\in \bN$ there exist  a sequence of unitaries  $(u_{n,k})_k\in \cU(A_{2})$, and  self-adjoint linear subspaces $(\widehat{F}_{n,k})_k\subseteq A_{2}$, with $1\in \widehat{F}_{n,k}$, which satisfy:
\begin{itemize}
\item (almost orthogonality) 
\begin{align*}
\sup_{a_{1},a_{2}\in V_{n,2}:\|a_{1}\|_{2},\|a_{2}\|_{2}\leq 1}|\tau(a_{1}u_{n,k}a_{2}u_{n,k})|&\to_{k\to\infty}0,
\\
\sup_{a_{1}\in V_{n,2},a_{2}\in \widehat F_{n,k}:\|a_{1}\|_{2},\|a_{2}\|_{2}\leq 1 }|\tau(a_{1}u_{n,k}a_{2})|&\to_{k\to\infty} 0,
\end{align*}

\item (inflated rapid decay) there are constants $b,\beta>0$ (independent of $n$) such that 
\[
\|x\|\leq b(n+1)^{\beta}\|x\|_{2}
\]
for all $x\in \widehat{F}_{n,k}$ and all $k$.

\item (asymptotic containment), for all $y\in V_{n,2}\ominus \bC 1$ we have 
    \[\lim_{k\to\infty}\max(\inf_{c\in \widehat{F}_{n,k}}\|y-c\|,\inf_{c\in \widehat{F}_{n,k}}\|u_{n,k}^{*}yu_{n,k}-c\|)=0. \]
\end{itemize}
Then $A_{1}*A_{2}$ is selfless. 
\end{introthm}


Let $M$ be the GNS completion of $A_{2}$ with respect to $\tau_{2}$. We show in Section \ref{fun continues} that if either $M$ has diffuse central sequence algebra or is a $\textrm{II}_{1}$-factor, then the bullet point conditions in Theorem \ref{thmE} hold with suitable unitaries $u_{n,k}\in \mathcal U(A_{n,2})$ and subspaces $\widehat F_{n,k}\subseteq A_2$. In this way we derive Theorem \ref{mainthm3}.

In the case that $M$ has diffuse central sequence algebra, we start with a normal embedding of 
$L^\infty([-1/2,1/2])$ into  the central sequence algebra of $M$. From this, we build unitaries $u_{n,k}\in A_2$ that satisfy the asymptotic orthogonality conditions above and asymptotically commute with $V_{n,2}$. We then set
\[
\widehat{F}_{n,k}=V_{n,2}\ominus \bC 1\quad\hbox{ for all }k,
\]
and the inflated rapid decay (comparing the operator and $2$-norms in  $\widehat F_{n,k}$) follows from the assumed rapid decay property of $V_{n,2}$.

In the case that $M$ is a $\textrm{II}_{1}$-factor, we 
start with a Haar unitary freely independent from $M$ in its tracial ultrapower, obtained from Popa's theorem \cite{popa1995free}. From this, we build  unitaries $u_{n,k}$ that satisfy the asymptotic orthogonality conditions and are such that $u_{n,k}^{*}(V_{n,2}\ominus \bC 1)u_{n,k}$ and $V_{n,2}\ominus \bC 1$ are asymptotically orthogonal. We then set
\[
\widehat{F}_{n,k}=u_{n,k}^{*}(V_{n,2}\ominus \bC 1)u_{n,k}+V_{n,2}\ominus \bC 1\quad\hbox{ for all }k.
\]
We use  asymptotic orthogonality of  $u_{n,k}^{*}(V_{n,2}\ominus \bC 1)u_{n,k}$ and $V_{n,2}\ominus \bC 1$  to check the inflated rapid decay for $\widehat{F}_{n,k}$.

\subsection*{Acknowledgements:}
It is our pleasure to thank Jamie Gabe, Greg Patchell, Chris Phillips, Chris Schafhauser, Itamar Vigdorovich, and Stuart White for helpful conversations. The first named author acknowledges support from the NSF CAREER award DMS-21447. The second author acknowledges support from NSF award DMS-2350049.

\section{Preliminaries}

\subsection{Preliminary notation}
Given $r\in \bN$, we use $[r]=\{1,\cdots,r\}$. 
By a $C^*$-probability space we understand  a pair 
$(A,\rho)$ where $A$ is a unital $C^*$-algebra and $\rho$ a state on $A$. If $\rho$ is a trace, we call  $(A,\rho)$ a tracial $C^*$-probability space. For $a\in A$, we use $\|a\|_{2}=\rho(a^{*}a)^{1/2}$. We will sometimes denote this by $\|a\|_{L^{2}(\rho)}$ if it is necessary to avoid confusion, but since $\rho$ will usually be clear from context we will typically just write $\|a\|_{2}$.
We also denote by 
$\|\cdot\|_2$ the Hilbert space norm in the GNS representation induced by $\rho$.

If $(A_{j},\rho_{j}),j\in [m]$ are $C^{*}$-probability spaces and the GNS representation of $A_{j}$ with respect to $\rho_{j}$ is faithful, let $(A_{1},\rho_{1})*\cdots*(A_{m},\rho_{m})$ be their free product, and $\rho_{1}*\cdots*\rho_{m}$ be the corresponding state on the free product.
If $\rho_{1},\cdots,\rho_{m}$ are clear from their context, we will often instead use $A_{1}*A_{2}*\cdots*A_{m}$. For $1\leq j\leq m$, we will use $\iota_{j}\colon A_{j}\to A_{1}*A_{2}*\cdots*A_{m}$ for the natural inclusion of $A_{j}$ into the $j^{th}$ piece of the free product.
 If $B_{i}\subseteq A_{i}$ are unital $*$-algebras, then we use $B_{1}*_{\textnormal{alg}}B_{2}*_{\textnormal{alg}}\cdots*_{\textnormal{alg}}B_{m}$ for the $*$-subalgebra of $A_{1}*A_{2}*\cdots*A_{m}$ generated by $\iota_{j}(B_{j}),j\in [m]$.
Often our arguments will involve maps from algebras of the form $A_{1}*_{\textnormal{alg}}A_{2}*_{\textnormal{alg}}A_{1}\to A_{1}*A_{2}$, in this case we slightly generalize the preceding notation to use $\iota_{j}\colon A_{s}\to A_{1}*A_{2}*A_{1}$ when $j\in [3],s\in [2]$ have the same parity (i.e. $s=1$ if $j\in \{1,3\}$ and $s=2$ if $j=2$) to mean the natural inclusion of $A_{s}$ into the $j^{th}$ piece of the free product.

Recall that if $r\in \bN$ and $\Omega$ is a set, then a function $j\colon [r]\to \Omega$ is \emph{alternating} if $j(i)\ne j(i+1)$ for all $1\leq i\leq r-1$. If $(A,\rho)$ is a $C^{*}$-probability space, and $S_{1},\cdots,S_{m}\subseteq A$ we say that $a=a_{1}\cdots a_{\ell}$ is an \emph{alternating word} in $(S_{1},\cdots,S_{m})$ if:
$a_{i}\in S_{j(i)}$ for an alternating function $j\colon [\ell]\to [m]$. If, in addition, $\rho(a_{i})=0$ for all $i$, then we call $a_{1}\cdots a_{\ell}$ an \emph{alternating centered word in $(S_{1},\cdots,S_{m})$.}
Note that, by the description of the free product Hilbert space, if $x\in (A_{1},\rho_{1})*\cdots*(A_{m},\rho_{m})$ is an alternating centered word in $(A_{1},\cdots,A_{m})$, then the decomposition of $x$ as an alternating centered word in $(A_{1},\cdots,A_{m})$ is unique modulo scaling.
Sometimes, especially in Section \ref{sec: here is where the fun begins}
, we will deal with alternating words that are not centered. In this case, the decomposition as an alternating word is not unique. Care must be taken in such cases, but this will not present an issue in our proofs.
In our analysis of alternating centered words, we will occasionally have to deal with words of length zero. As we will usually be in the unital case, we will take the convention that the empty product in a unital algebra is equal to $1$. Similarly, we take the convention that the empty sum is zero.

Throughout we use $A\lesssim B$ to mean that $A$ is less than or equal to a universal constant times $B$. If we wish to stress that the constant depends upon some ambient parameters, we will write $A\lesssim_{\alpha}B$ to mean that $A$ is at most a constant times $B$, where the constant is allowed to depend upon $\alpha$ (similarly $A\lesssim_{\alpha,\beta}B$ means $A$ is at most a constant times $B$, where the constant depends upon $\alpha$ and $\beta$).

\subsection{Selfless $C^*$-algebras}

Let $\omega$ be a free ultrafilter on a 
set $I$. For a $C^{*}$-algebra $A$, we let $A^{\omega}$ be the $C^{*}$-ultrapower of $A$ with respect to $\omega$. If $\rho$ is a state on $A$,
we let $\rho^\omega$ be the state on $A^\omega$
obtained by taking the limit of $\rho$ along $\omega$,
i.e., $\rho^\omega(a)=\lim_{i\to\omega} \rho(a_i)$, where $a\in A^{\omega}$ is the image of $(a_i)_i$ in $A^{\omega}$.

Recall that a $C^{*}$-probability space $(A,\rho)$ such that $\rho$ has faithful GNS representation is called \emph{selfless} (see \cite{robert2023selfless}) if there is another $C^{*}$-probability space $(C,\kappa)$ with $\dim_{\bC}(C)>1$ and an injective $*$-homomorphism  $\theta\colon (A,\rho)*(C,\kappa)\hookrightarrow A^{\omega}$ for some free ultrafilter $\omega$ on a direct set $I$, such that $\theta|_{A}$ is the diagonal embedding and $\rho^{\omega}\circ \theta=\rho*\kappa$.
Selfless $C^{*}$-algebras have many desirable properties such as simplicity, real rank zero if $\rho$ is faithful but not a trace, stable rank one if $\rho$ is a trace, uniqueness of the trace if $\rho$ is a trace, as well as strict comparison if $\rho$ is a trace.

The main criterion for selflessness of $C^{*}$-algebras we will use is the following.

\begin{prop}\label{wk convergence+ strong semi conv}
Let $(A,\rho_{A})$ be a unital  $C^{*}$-probability space with faithful GNS representation. Suppose that there is another $C^{*}$-probability space $(B,\rho_{B})$ with faithful GNS representation, 
and such that $B\neq \bC$. Assume we are given a dense, unital $*$-subalgebra $C\subseteq A*B$ with $C\cap A$ dense in $A$, and linear $*$-preserving maps $\phi_{k}\colon C\to A$ for $k=1,2,\ldots$, such that 
\begin{itemize}
\item $\|\phi_{k}(ab)-\phi_{k}(a)\phi_{k}(b)\|\to_{k\to\infty}0$ for all $a,b\in C$,
     \item $\phi_{k}|_{C\cap A}$ is the identity,
    \item $\rho_{A}\circ \phi_{k}\to (\rho_{A}*\rho_{B})|_{C}$ pointwise,
    \item $\limsup_{k\to\infty}\|\phi_{k}(a)\|\leq \|a\|$ for all $a\in C$.
\end{itemize}
Then $(A,\rho_{A})$ is selfless.
\end{prop}

\begin{proof}
Let $\omega$ be a free ultrafilter on $\bN$. Let $A^{\omega}$ be the $C^{*}$-ultrapower of $A$, and define $\phi\colon C\to A^{\omega}$ by sending $a$ to the image of the sequence $(\phi_{k}(a))_{k}$. Then for all $a\in C$,
\[\|\phi(a)\|=\lim_{k\to\omega}\|\phi_{k}(a)\|\leq \|a\|.\]
Thus $\phi$ extends to a $*$-homomorphism $A*B\to A^{\omega}$, which we still denote by $\phi$.  Since $\phi$ is a $*$-homomorphism we know $\|\phi(a)\|\leq \|a\|$ for all $a\in A*B$.
By the second bullet point above, we know that  $\phi$ is the identity on $A$ (since $C\cap A$ is dense in $A$). By the third bullet point above, we know that $\rho_{A}^{\omega}\circ \phi=\rho_{A}*\rho_{B}$. 
Since $\rho_{A}*\rho_{B}$ induces a faithful GNS representation of $A*B$, we have for every $a\in A*B$:
\begin{align*}
\|a\|&=\sup_{x,y\in A*B:\|x\|_{2},\|y\|_{2}\leq 1}(\rho_{A}*\rho_{B})(y^{*}ax)\\
&=\sup_{x,y\in A*B:\|x\|_{2},\|y\|_{2}\leq 1}\rho_{A}^{\omega}(\phi(y)^{*}\phi(a)\phi(x))\leq \|\phi(a)\|.
\end{align*}

Thus $\phi$ is operator-norm isometric. So we have shown that $A*B\hookrightarrow A^{\omega}$ is faithful. 
\end{proof}

\section{Rapid decay property for filtrations and tuples}\label{sec: rapid decay}

\begin{defn}\label{defn:RDP}
Let $(A,\rho)$ be a unital $C^{*}$-probability space. A \emph{filtration} of $A$ is an increasing sequence $(V_{n})_{n=0}^{\infty}$ of linear subspaces of $A$ such that:
\begin{itemize}
    \item $V_{0}=\mathbb C 1$,
    \item $V_{n}$ is stable under $*$,
    \item $V_{n}V_{k}\subseteq V_{n+k}$ for all $n,k$,
    \item $\overline{\bigcup_{k}V_{k}}^{\|\cdot\|}=A$.

\end{itemize}
\end{defn}

Note that if we have a sequence $(V_{n})_{n=0}^{\infty}$ of linear subspaces satisfying the first three bullet points, then automatically $\overline{\bigcup_{k}V_{k}}^{\|\cdot\|}$ is a $C^{*}$-subalgebra. So the last bullet point is not a major assumption and can always be reduced to without loss of generality. 

One key example of interest is the following. Suppose that $x=(x_{1},\cdots,x_{r})\in A^{r}$ and that $A=C^{*}(x,1)$. In this case, we obtain a filtration on $A$ by setting 
\[
V_{n}=\{P(x):P\in \bC^{*}\ip{T_{1},\cdots,T_{r}} \textnormal{ and } \deg(P)\leq n\}
\]
for $n=0,1,\ldots$. We call $(V_n)_n$ the natural filtration obtained from the generating tuple $x$.

\subsection{The rapid decay property}

\begin{defn}
Given a $C^*$-probability space $(A,\rho)$ with filtration $(V_{n})_{n}$, we say that $(V_{n})_{n}$ has the \emph{rapid decay property}, or that $(A,\rho)$  has rapid decay relative to $(V_{n})_{n}$, if 
there is an $\alpha>0$ such that 
    \[
    \|a\|\lesssim_{\alpha} (n+1)^{\alpha}\|a\|_{2}
    \]
    for all $a\in V_{n}$.
  When $(V_n)_n$ is the degree-filtration of $C^{*}(x,1)$ coming from a tuple $x\in A^r$, we say that the tuple $x$ has rapid decay.  
\end{defn}

We note that the existence of a filtration with the rapid decay property implies faithfulness of the corresponding GNS representation. This is particularly relevant for us, since we need to assume faithful GNS representation in order to deduce selflessness.

\begin{prop}\label{RDfaithfulGNS}
Let $(A,\rho)$ be a unital $C^{*}$-probability space which has a filtration $(V_{n})_{n=0}^{\infty}$ with the rapid decay property. Then the GNS representation of $\rho$ is faithful.     
\end{prop}

\begin{proof}
Choose $c,\alpha>0$ such that $\|x\|\leq c(n+1)^{\alpha}\|x\|_{2}$ for all $x\in V_{n}$ and all $n$. 
Let $\pi_{\rho}$ be the GNS representation of $\rho$. Fix $n\in \bN$ and let $a\in V_{n}$. Then for any $r\in \bN$ we have that $(a^{*}a)^{r}\in V_{2nr}$. Thus
\[\|a\|^{2r}=\|(a^{*}a)^{r}\|\leq c(2nr+1)^{\beta}\|(a^{*}a)^{r}\|_{2}\leq c(2nr+1)^{\beta}\|\pi_{\rho}((a^{*}a)^{r})\|,\]
the last inequality following from the fact that 
\[\|(a^{*}a)^{r}\|_{2}=\|\pi_{\rho}((a^{*}a)^{r})1\|_{L^{2}(\rho)}.\]
Thus
\[\|a\|\leq c^{1/2r}(2nr+1)^{\beta/2r}\|\pi_{\rho}(a)\|.\]
Letting $r\to\infty$ we see that $\pi_{\rho}$ is isometric on $\bigcup_{n}V_{n}$. 
By the density of $\bigcup_{n}V_{n}$ in $A$, we see that $\pi_{\rho}$ is isometric and thus faithful. 
\end{proof}

Our next goal is to adapt arguments of Jolissaint \cite{JoliRDP} to obtain several equivalent ways of characterizing the rapid decay property. It will be helpful to isolate a key result, which will also be used at a few points throughout the paper.

\begin{lem}\label{lem: passing to HS closure}
Let $(A,\rho)$ be a unital $C^{*}$-probability space. Suppose that $V\subseteq A$ is a linear subspace and $C\geq 0$ with $\|a\|\leq C\|a\|_{2}$ for all $a\in V$. Then, $\overline{V}^{\|\cdot\|}$ is complete with respect to $\|\cdot\|_{2}$, $\overline{V}^{\|\cdot\|_{2}}=\overline{V}^{\|\cdot\|}$, and $\|a\|\leq C\|a\|_{2}$ for all $a\in \overline{V}^{\|\cdot\|}$.  
\end{lem}

\begin{proof}
Suppose that $a_{n}\in \overline{V}^{\|\cdot\|}$ is Cauchy with respect to $\|\cdot\|_{2}$. Choose $\widetilde{a}_{n}\in V$ with $\|\widetilde{a}_{n}-a_{n}\|\to_{n\to\infty}0$. Since $\|\cdot\|_{2}\leq \|\cdot\|$, we know that $\widetilde{a}_{n}$ is Cauchy with respect to $\|\cdot\|_{2}$. Since $\|x\|\leq C\|x\|_{2}$ for all $x\in V$, we have that $(\widetilde{a}_{n})_{n=0}^{\infty}$ is Cauchy in $A$. Thus there is $a\in \overline{V}^{\|\cdot\|}\subseteq A$ with $\|\widetilde{a}_{n}-a\|\to_{n\to\infty}0$. Since $\|\cdot\|_{2}\leq \|\cdot\|$, it follows that $\|a_{n}-a\|_{2}\to_{n\to\infty}0$. This proves that $\overline{V}^{\|\cdot\|}$ is $L^{2}$-complete and that  $\overline{V}^{\|\cdot\|_{2}}=\overline{V}^{\|\cdot\|}\subseteq A$. Finally, we have that 
\[\|a\|=\lim_{n\to\infty}\|\widetilde{a}_{n}\|\leq C\lim_{n\to\infty}\|\widetilde{a}_{n}\|_{2}=\|a\|_{2}.
\qedhere\]
\end{proof}

For faithful $C^{*}$-probability spaces, we can characterize the rapid decay property via locally convex spaces as follows.

\begin{defn}
Let $(A,\rho)$ be a unital $C^{*}$-probability space, 
with $\rho$ faithful.
By faithfulness, we may identify $A\subseteq L^{2}(A,\rho)$. 

Let $(V_n)_n$ be a filtration on $A$. For $k\in \bN$, let $q_{k}\in B(L^2(A,\rho))$ be the projection onto $\overline{V_{k}}^{\|\cdot\|_2}$, and set $p_{k}=q_{k}-q_{k-1}$. For $y\in L^{2}(A,\rho)$ and $\alpha>0$ we define
\begin{align*}
\|y\|_{S^{\infty,\alpha}} &=\sum_{k=0}^{\infty}(k+1)^{\alpha}\|p_{k}(y)\|,\\
\|y\|_{S^{2,\alpha}} &=\left(\sum_{k=0}^{\infty}(k+1)^{2\alpha}\|p_{k}(y)\|_{2}^{2}\right)^{1/2}.
\end{align*}

Define
vector subspaces
\begin{align*}
S^{\infty,\alpha}((V_{n})_{n}) &=\{y\in L^{2}(A,\rho):\|y\|_{S^{\infty,\alpha}}<+\infty\},\\
S^{2,\alpha}((V_{n})_{n}) &=\{y\in L^{2}(A,\rho):\|y\|_{S^{2,\alpha}}<+\infty\}
\end{align*}
equipped with the norms 
$\|\cdot\|_{S^{\infty,\alpha}}$ and $\|\cdot\|_{S^{2,\alpha}}$, respectively. If $\alpha>1/2$, then it can be shown that both $S^{\infty,\alpha}((V_{n})_{n})$ and $S^{2,\alpha}((V_{n})_{n})$ are complete.
We let 
\begin{align*}
S^{\infty}((V_{n})_{n}) &=\bigcap_{\alpha>0}S^{\infty,\alpha}((V_{n})_{n}),\\
S^{2}((V_{n})_{n}) &=\bigcap_{\alpha>0}S^{2,\alpha}((V_{n})_{n}).
\end{align*}
\end{defn}

We give $S^{\infty}((V_{n})_{n})$ and $S^{2}((V_{n})_{n})$ a locally convex structure by equipping them with the family of norms $(\|\cdot\|_{S^{\infty,\alpha}})_{\alpha>0}$ and $(\|\cdot\|_{S^{2,\alpha}})_{\alpha>0}$, respectively. Restricting these families of seminorms to $\alpha\in \bN$ defines the same topology and thus these locally convex topologies are metrizable. These topologies are also complete (the verification is straightforward and is left to the reader).
In \cite[Remark 5.10, Proposition 5.11]{VoiculescuFiltrations}, Voiculescu gives multiple characterizations of $S^{\infty}((V_{n})_{n})$.

\begin{prop}\label{prop: TFAE RDP}
Let $(A,\rho)$ be a unital $C^{*}$-probability space
with $\rho$ faithful. Let $(V_{n})_{n}$ be a filtration of $A$. Then the following are equivalent:
\begin{enumerate}[(i)]
    \item $(A,\rho)$ has the rapid decay  property relative to $(V_n)_n$.     \label{item:RDP estimate}

    \item there is an $\alpha>0$ such that
    \[\|ab\|_{2}\lesssim_{\alpha} (n+1)^{\alpha}\|a\|_{2}\|b\|_{2}\]
    for all $a\in V_{n}$, $b\in \bigcup_{k}V_{k}$. 
        \label{item:RDP bilinear estimate} 

 \item  $p_n(V_n)\subseteq A$  and there is a $\beta>0$ such that 
    \[\|p_{n}(a)\|\lesssim_{\beta} (n+1)^{\beta}\|a\|_{2}\]
    for all $a\in V_{n}$. 
    \label{item:RDP estimate on layers}

    \item $p_n(V_n)\subseteq A$ and there is a $\beta>0$ such that 
     \[\|p_{n}(a)b\|_{2}\lesssim_{\beta} (n+1)^{\beta}\|a\|_{2}\|b\|_{2}\]
    for all $a\in V_{n}$, $b\in \bigcup_{k}V_{k}$. \label{item:RDP bilinear estimate on layers}

   \item $S^{2}((V_{n})_{n})\subseteq A$. 
   \label{item:RPD implies bounded}

\end{enumerate}
\end{prop}

\begin{proof}
The equivalence of (\ref{item:RDP estimate}) and (\ref{item:RDP bilinear estimate}) follows from the density of $\bigcup_{n}V_{n}$ in $A$. 
The equivalence of (\ref{item:RDP estimate on layers})
and (\ref{item:RDP bilinear estimate on layers}) also follows from the density of $\bigcup_{n}V_{n}$ in $A$. 

To see that (\ref{item:RDP estimate}) implies 
(\ref{item:RDP estimate on layers}), note first that by 
Lemma \ref{lem: passing to HS closure}, $p_n(V_n)\subseteq A$, and that the inequality $\|a\|\lesssim_\alpha (n+1)^\alpha \|a\|_2$ remains valid in $\overline{V_n}^{\|\cdot\|_2}$. Now  (\ref{item:RDP estimate on layers})
follows immediately using that $\|p_n(a)\|_2\leq \|a\|_2$.

To see that (\ref{item:RDP estimate on layers})
implies (\ref{item:RDP estimate}),  let $a\in V_{n}$, and set $a_{\ell}=p_{\ell}(a)$ for $\ell\in \bN\cup\{0\}$.  Then
\begin{align*}
 \|a\|\leq \sum_{\ell=0}^{n}\|a_{\ell}\|\lesssim \sum_{\ell=0}^{k}(\ell+1)^{\beta}\|a_{\ell}\|_{2}
 &\leq \|a\|_{2}\left(\sum_{\ell=0}^{n}(\ell+1)^{2\beta}\right)^{1/2}\\
 &\leq (n+1)^{\beta+1/2}\|a\|_{2}.
\end{align*}

To see that (\ref{item:RPD implies bounded}) implies (\ref{item:RDP estimate}), note that $S^{2}((V_{n})_{n})$ is a locally convex Frechet space. By a version of the closed graph theorem valid for Frechet spaces (see \cite[Theorem 2.15]{RudinFA}, we see that the inclusion map $\iota\colon S^{2}((V_{n})_{n})\to A$ is necessarily continuous. By the continuity criteria for linear maps between locally convex spaces (\cite[Proposition 5.15]{ThaBible}), we see that there is an $\alpha>0$ such that 
\[\|y\|\lesssim_{\alpha} \|y\|_{S^{2,\alpha}}\]
for all $y\in S^{2}((V_{n})_{n})$. Suppose $a\in V_{n}$. Then $p_{k}(a)=0$ for all $k>n$. Thus:
\[\|a\|^{2}\lesssim_{\alpha} \sum_{k=0}^{n}\|p_{k}(a)\|_{2}^{2}(k+1)^{2\alpha}
\leq (n+1)^{2\alpha}\sum_{k=0}^{n}\|p_{k}(a)\|_{2}^{2}\\
 =(n+1)^{2\alpha}\|a\|_{2}^{2}.\]
Taking square roots completes the proof. 

Finally, let us show that  (\ref{item:RDP estimate}) implies (\ref{item:RPD implies bounded}). 
Fix $y\in S^{2}((V_{n})_{n})$.
From Lemma \ref{lem: passing to HS closure} we know that
$p_{k}(y)\in A$ for all $k$, and that
\[
\|p_{k}(y)\|\lesssim_{\alpha}(k+1)^{\alpha}\|p_{k}(y)\|_{2}.
\]
For $K\in \bN$, set 
\[y_{K}=\sum_{k=0}^{K}p_{k}(y).\]
 Note that $y_{K}\in A$ for all $K$. For any pair of integers $L\geq K$ we have:
\begin{align*}
\|y_{L}-y_{K}\|\leq \sum_{k=K+1}^{L}\|p_{k}(y)\|_{\infty}&\leq \sum_{k=K+1}^{L}\|p_{k}(x)\|_{2}(k+1)^{\alpha}\\
&\lesssim_{\alpha} \|y\|_{S^{2,\alpha+1}}\left(\sum_{k=K+1}^{\infty}\frac{1}{(k+1)^{2}}\right)^{1/2}.    
\end{align*}
Thus
\[\lim_{K\to\infty}\sup_{L\geq K}\|y_{L}-y_{K}\|=0,\]
and completeness of $A$ implies that there is a $\hat{y}\in A$ with $\|y_{K}-\hat{y}\|\to_{K\to\infty}0$. Since $\|y-y_{K}\|_{2}\to_{K\to\infty}0$, we see that $\hat{y}=y$ as elements of $L^{2}(A,\tau)$. Thus $y\in A$. 
\end{proof}

The following corollary readily follows 
from Lemma \ref{lem: passing to HS closure}, but we include it for completeness.

\begin{cor}\label{cor: complete the filtration}
Let $(A,\rho)$ be a unital $C^{*}$-probability space, and let $(V_{n})_{n}$ be a filtration of $A$ with the rapid decay property. Then:
\begin{enumerate}[(i)]
\item $\overline{V_{n}}^{\|\cdot\|_{2}}\subseteq A$,
\item $\overline{V_{n}}^{\|\cdot\|_{2}}$ is a filtration of $A$ with the rapid decay property. 
\end{enumerate}
\end{cor}

\begin{proof}
 The first item explicitly follows from Lemma \ref{lem: passing to HS closure}. Lemma \ref{lem: passing to HS closure} states also that $\overline{V_{n}}^{\|\cdot\|}=\overline{V_{n}}^{\|\cdot\|_{2}}$. Since the norm closure of a filtration is directly seen to be a filtration, the second item follows. 
\end{proof}

Let us examine some preservation properties of the property of rapid decay.
\begin{prop}\label{RDpreservation}
Let $(A_1,\rho_1)$ and $(A_2,\rho_2)$ be $C^*$-probability spaces with rapid decay relative to filtrations $(V_{n,1})_n$ and $(V_{n,2})_n$. 
\begin{enumerate}[(i)]
\item
If $\alpha>0$, then 
\[
(A_1\oplus A_2,\alpha\rho_1+(1-\alpha)\rho_2)
\]
has rapid decay relative to the filtration
$V_0=\bC 1$, $V_n=V_{n,1}\oplus V_{n,2}$ for $n\geq 1$.
\item 
If $p\in \bigcup_n V_{n,1}$ is  a projection such that
$\rho_1(p)>0$, the  $(pA_1p,\frac{1}{\rho_1(p)}\rho_1)$ has rapid decay relative to the filtration $V_{n,1}'=pA_1p\cap V_{n,1}$.

\item  \label{item: tesnor preserves RDP partial}
Suppose that the following condition holds: for all $n\in \bN$, $k\in \bN$,  and  all $s\in \{n,k,n+k\}$, there are  orthonormal bases $(x_{i,s})_{i\in I_{s}}$ for $\overline{V_{s,2}}^{\|\cdot\|_{2}}$ satisfying that $\ip{x_{i,n}x_{j,k},x_{p,n+k}}\geq 0$  for all $i\in [n]$, $j\in [k]$, and 
$p\in [n+k]$.
Then 
\[
(A_1\otimes A_2,\rho_1\otimes \rho_2)
\]
has rapid decay
relative to the filtration $(V_{n,1}\otimes V_{n,2})_n$, where $A_1\otimes A_2$ denotes the minimal tensor product.
\end{enumerate}
\end{prop}
\begin{proof}
(i) This readily follows from $\|(a_1,a_2)\|=\max(\|a_1\|,\|a_2\|)$ and 
\[
\max(\alpha^{\frac12}\|a_1\|_2, (1-\alpha)^{\frac12}\|a_2\|_2)\leq \|(a_1,a_2)\|_2.
\]

(ii) Note that $(V_n')_n$ indeed is a filtration of $pA_1p$ (use $p\in \bigcup V_{n,1}$ to show the density of $\bigcup_n V_{n,1}'$ in $pA_1p$). If $b\in V_{n}'$, then
\[
\|b\|\lesssim_{\alpha} (1+n)^{\alpha}\|b\|_{2,\rho_1}=
(1+n)^{\alpha}\rho_1(p)^{\frac12}\|b\|_{2,\tilde\rho_1},
\]
where $\tilde \rho_1=\frac{1}{\rho_1(p)}\rho_1$.

(iii)  By Corollary \ref{cor: complete the filtration}, we may assume that $V_{n,2}$ is $\|\cdot\|_{2}$-closed for all $n$. Below $\odot$ denotes the algebraic tensor product and $\otimes$ the minimal one. 
Choose $\alpha_{j}>0$, with $j=1,2$, so that 
\[
\|x\|\lesssim (n+1)^{\alpha_{j}}\|x\|_{2}, \textnormal{ for all $x\in V_{n,j}$ and $j=1,2$.}
\]
Fix $n,k\in \bN$, and let $(x_{i,s})_{s\in I_{s}}$ be orthonormal bases as in the assumption for $s\in \{n,k,n+k\}$.
Set $E_{s,2}=\Span\{x_{i,s}:i\in I_{s}\}$.
Let $c\in V_{n,1}\odot E_{n,2}$ and $d\in V_{k,1}\odot E_{k,2}$. Write 
\[
c = \sum_{i\in I_{n}} c_{i}\otimes x_{i,n},\quad d = \sum_{j\in I_{k}} d_{j}\otimes x_{j,k}.
\]
Let $S_{c}=\{i:c_{i}\ne 0\}$, $S_{d}=\{j:d_{j}\ne 0\}$, and note that both these sets are finite. 
For $p\in [n+k]$, let $\lambda_{i,j,p}=\ip{x_{i,n}x_{j,k},x_{p,n+k}}$. Then:
\[cd=\sum_{p}\left(\sum_{(i,j)\in S_{c}\times S_{d}}\lambda_{i,j,p}c_{i}d_{j}\right)\otimes x_{p,n+k}.\]
Since $S_{c},S_{d}$ are finite sets, there are no convergence issues when we write the sum above, and the sum over $p$ converges in $\|\cdot\|_{2}$. Set 
\[
\Theta_{c}=\sum_{i\in S_{c}}\|c_{i}\|_{2}x_{i,n}, \quad \Theta_{d}=\sum_{j\in S_{d}}\|d_{j}\|_2x_{j,k}.
\]
Using that $\lambda_{i,j,p}\geq 0$ for all $i\in [n],j\in [k],p\in [n+k]$, we have:
\begin{align*}
\|cd\|_{2}^{2}=\sum_{p}\left\|\sum_{(i,j)\in S_{c}\times S_{d}}\lambda_{i,j,p}c_{i}d_{j}\right\|_{2}^{2}&\leq   \sum_{p}\left(\sum_{(i,j)\in S_{c}\times S_{d}}\lambda_{i,j,p}\|c_{i}d_{j}\|_{2}\right)^{2}\\
&\lesssim (n+1)^{2\alpha_{1}}\sum_{p}\left(\sum_{(i,j)\in S_{c}\times S_{d}}\lambda_{i,j,p}\|c_{i}\|_{2}\|d_{j}\|_{2}\right)^{2}\\
&=(n+1)^{2\alpha_{1}}\|\Theta_{c}\Theta_{d}\|_{2}^{2},
\end{align*}
with the last inequality following by the assumed rapid decay property of $A_{1}$ with respect to $(V_{\ell,1})_{\ell}$.
By the assumed rapid decay property of $A_{2}$ with respect to $(V_{t,2})_{t}$, we have that 
\[\|\Theta_{c}\Theta_{d}\|_{2}\lesssim (n+1)^{\alpha_{2}}\|\Theta_{c}\|_{2}\|\Theta_{d}\|_{2}=(n+1)^{\alpha_{1}}\|c\|_{2}\|d\|_{2}.\]
Thus 
\[
\|cd\|_{2}\lesssim (n+1)^{\alpha_{1}+\alpha_{2}}\|c\|_{2}\|d\|_{2}
\]
for all $c\in V_{n,1}\odot E_{n,2},d\in V_{k,1}\odot E_{k,2}$. By continuity, the same estimate holds for $c\in \overline{V_{n,1}\odot E_{n,1}}^{\|\cdot\|}$ and  $d\in \overline{V_{k,1}\odot E_{n,2}}^{\|\cdot\|}$. By Lemma \ref{lem: passing to HS closure}, it follows that the above estimate holds for $c\in V_{n,1}\otimes V_{n,2}$ and $d\in V_{k,1}\otimes V_{k,2}$. We thus conclude by applying Proposition \ref{prop: TFAE RDP}. 
\end{proof}

Note that the hypotheses of (\ref{item: tesnor preserves RDP partial}) hold if
$A_{2}=C^{*}_{\lambda}(G)$, where $G$ is 
a finitely generated group, or more generally 
a finitely generated discrete hypergroup (in the sense of \cite{MR2777191}), $\rho_2$ is the canonical trace on $C_\lambda^*(G)$, and $(V_{n,2})_n$ is the filtration obtained from a finite generating set (see Example \ref{RDgroups}).

We end this subsection by noting that we have a well-defined way of saying that a finitely generated  $*$-algebra with a state has the rapid decay property.

\begin{prop}
Let $(A,\rho)$ be a unital $C^{*}$-probability space, and $x\in A^{r}$ with $A=C^{*}(x,1)$. If $x$ has the rapid decay property, then every finite tuple in the $*$-algebra generated by $x$ has the rapid decay property.    
\end{prop}

\begin{proof}
 Let $B$ be the $*$-algebra generated by $x$, and let $y=(y_{1},\cdots,y_{s})\in B^{s}$. Write $y_{j}=Q_{j}(x)$ for some $*$-polynomial $Q_{j}$ in $r$ noncommuting variables. If $P$ is a noncommutative $*$-polynomial in $s$ variables, then $F=P\circ (Q_{1},\cdots,Q_{s})$ is a noncommutative $*$-polynomial in $r$ variables. Let $d$ be the maximum of the degrees of the $Q_{j}$'s .Then for some constant $\alpha>0$ we have
 \begin{align*}
 \|P(y)|\|=\|F(x)\|\lesssim (\deg(F)+1)^{\alpha}\|F(x)\|_{2}&=(\deg(F)+1)^{\alpha}\|P(y)\|_{2}\\
 &\leq (\deg(P)d+1)^{\alpha}\|P(y)\|_{2}\\
& \leq d^{\alpha}(\deg(P)+1)^{\alpha}\|P(y)\|_{2}.    
\qedhere \end{align*}
\end{proof}

Because of the above corollary, we may say that a finitely generated  $*$-algebra $B$ equipped with a state $\rho$  has the \emph{rapid decay property} if for some (equivalently any) finite generating tuple $x$, we have that $x$ has the rapid decay property when viewed in the $C^{*}$-completion under the GNS representation of $B$ with respect to $\rho$.

\subsection{Tracial polynomial growth}
Here we introduce a property stronger than rapid decay that
is nevertheless present in many examples of interest.

\begin{prop}\label{polygrowthTFAE}
Let $(A,\tau)$ be a faithful tracial $C^*$-probability space and $(V_n)_{n\ge0}$ a  filtration  with $\dim V_n<\infty$ for all $n$. Then the following are equivalent:
	\begin{enumerate}[(i)]
	\item 
	There exists $\beta>0$ such that 
	\[
	\|aq_n\|_{HS}\lesssim  (1+n)^\beta\|a\|_2\hbox{ for all }a\in A.
	\]
    (Here $\|\cdot\|_{HS}$ denotes the Hilbert-Schmidt norm of operators on $L^2(A,\tau)$ and $q_n\in B(L^2(A,\tau))$ the orthogonal projection
	onto $\overline{V_n}^{\|\cdot\|_2}$.)
	
	\item 
	If $x_1,\ldots,x_{k_n}$ is an orthonormal basis (o.n.b.) of $V_n$, then 
	\[
	n\mapsto \Big\|\sum_{i=1}^{k_n} x_ix_i^*\Big\|
	\] 
	grows at most polynomially.
	
	\item 
	$n\mapsto \dim V_n$ grows at most polynomially and 
    $(A,\tau)$ has rapid decay relative to $(V_n)_{n=0}^\infty$. 
	
	\item 
	There exists $\alpha>0$ such that $S^{2,\alpha}((V_n)_n)\subseteq S^{\infty,0}((V_n)_n)$.
	
	\end{enumerate}	
\end{prop}

\begin{proof}
With $k_n=\dim V_n$ and $x_1,\ldots,x_{k_n}$ as in (ii), let
\[
d_n=\Big(\sum_{i=1}^{k_n} x_ix_i^*\Big)^{\frac12}.
\]
Then
\[
\|aq_n\|_{HS}^2=
\sum_{i=1}^{k_n} \tau(x_i^*a^*ax_i)=
\tau(a^*ad_n^2)=\|ad_n\|_2^2
\]
for all $a\in A$. That is, $\|aq_n\|_{HS}=\|ad_n\|_2$. From this we get that
\[
\sup_{a\in A\backslash\{0\}} \frac{\|aq_n\|_{HS}}{\|a\|_2}= \|d_n\|.
\] 
This show (i) $\Leftrightarrow$ (ii).

(ii) implies (iii). The polynomial growth of $\dim V_n$ follows from 
\[
(\dim V_n)^{\frac12} = \|d_n\|_2\leq  \|d_n\|.
\]	

For $a=\sum_{i=1}^{k_n} \alpha_i x_i\in V_n$,  we have
\[
\|a\| \leq \|a\|_2 \|d_n\|
\]
(using that $\left\|\sum_{i}b_{i}^{*}y_{i}\right\|\leq \left\|\sum_{i}b_{i}^{*}b_{i}\right\|^{1/2}\left\|\sum_{i}y_{i}^{*}y_{i}\right\|^{1/2}$ for $b_{1},\ldots,b_{n},y_{1},\ldots,y_{n}\in A$). It follows that $A$ has rapid decay.

(iii) implies (iv). Suppose that $A$ has rapid decay with exponent $\alpha$. Then
\[
\|d_n\|\lesssim (1+n)^\alpha \|d_n\|_2=(1+n)^\alpha (\dim V_n)^{\frac12}.
\]
Thus, $\|d_n\|$ has polynomial growth.

(ii) implies (iv). Write $a=\sum_n a_n$, with $a_n=p_n(a)\in V_n\ominus V_{n-1}$, and define $y_n=(d_n^2-d_{n-1}^2)^{\frac12}$ for all $n$. Then
\[
\|a\|_{\infty,0}=\sum_n \|a_n\| \leq \sum_n \|a_n\|_2 \|y_n\|.
\]
It follows that if $\|y_n\|\lesssim (1+n)^\alpha$, then $S^{2,\alpha}$ is contained in $S^{\infty,0}$.

(iv) implies (iii). Since $S^{\infty,0}((V_n)_n)\subseteq A$, 
we get that $(A,\tau)$ has rapid decay (by Proposition \ref{prop: TFAE RDP} (v)). In order to show that $\dim V_n$ grows polynomially, 
if suffices to show that $\rank (p_n)=\dim (V_n\ominus V_{n-1})$ grows polynomially.
Suppose, for the sake of contradiction, that $\rank(p_n)$ does not grow 
polynomially. Find a strictly increasing sequence $(n_k)_k$ such that $\mathrm{rank} (p_{n_k})\geq n_k^k$.
Choose any $1<p<2$, and define
\[
a=\sum_k \frac{1}{\|z_{n_k}\|^{\frac 1p}}z_{n_k},
\]
where $z_{n_k}=\sum_{n_k}^{n_k+1} x_i$. Since $\|z_{n_{k}}\|\geq \|z_{n_{k}}\|_{2}=\rank(p_{n_{k}})^{1/2}$, $a\notin S^{\infty,0}((V_n)_n)$. On the other hand,
\begin{align*}
\|a\|_{2,\alpha}^2 &=\sum_k (1+n_k)^\alpha \frac{1}{\|z_{n_k}\|^{\frac 2p}}\rank(p_{n_k})^2\\
&\leq  \sum_k (1+n_k)^\alpha \rank(p_{n_k})^{2-\frac2p}\\
&\leq\sum_k (1+n_k)^\alpha n_k^{k(2-\frac2p)} <\infty.
\end{align*}
That is, $a\in S^{2,\alpha}((V_n)_n)$. This yields the desired contradiction.
\end{proof}

\begin{defn}
We say that a tracial $C^*$-probability space $(A,\tau)$ has tracial polynomial growth relative to a filtration $(V_n)_{n}$ if any of the equivalent conditions of the previous proposition hold.	
\end{defn}

\begin{rem}    
A filtration $(V_n)_n$ of a $C^*$-algebra $A$ is said to satisfy Voiculescu's condition if
\[
\liminf \dim V_n^{\frac1n}=1.
\]	
(See \cite{Vaillant}). If  $A$ admits a filtration   
that satisfies Voiculescu's condition, then $A$ is nuclear, by \cite[Corollary 2.2]{KirchbergVaillant}, and  by \cite[Corollary 5.4]{VoiculescuFiltrations}
every quotient of $A$ has a tracial state, i.e.,
 $A$ has the QTS property. Thus, if $A$ has tracial polynomial growth, then $A$ has rapid decay, it is nuclear, and has the QTS property.
We note that, by \cite[Theorem 2]{Ozawa}, a $C^*$-algebra is nuclear and with the QTS property if and only if it is symmetrically amenable.  	
\end{rem}


Polynomial growth is stable under tensor products:

\begin{prop}\label{prop: poly growth tens}
Suppose that $(A,\rho_{A}),(B,\rho_{B})$ are $C^{*}$-probability spaces which have rapid decay relative to filtrations $(V_{n})_{n},(W_{n})_{n}$. Suppose that $\dim(V_{n})\lesssim (n+1)^{\gamma}$ for some $\gamma>0$. Then $A\otimes B$ has rapid decay relative to the filtration $V_{n}\otimes W_{n}$.
\end{prop}		

\begin{proof}
Choose $\alpha,\beta>0$ so that $\|x\|\lesssim (n+1)^{\alpha}\|x\|_{2}$ for all $x\in V_{n}$, and $\|y\|\lesssim (n+1)^{\beta}\|y\|_{2}$ for all $y\in W_{n}$.

Fix $n\in \bN$, and let $(x_{j})_{j=1}^{k_{n}}$ be an orthonormal basis for $V_{n}$. Given $a\in V_{n}\otimes W_{n}$, we can then write $a=\sum_{j=1}^{k_{n}}x_{j}\otimes a_{j}$ with $a_{j}\in W_{n}$. Then:
\[\|a\|\leq \sum_{j=1}^{k_{n}}\|x_{j}\|\|a_{j}\|\lesssim (n+1)^{\alpha+\beta} \sum_{j=1}^{k_{n}}\|a_{j}\|_{2}\leq (n+1)^{\alpha+\beta}k_{n}^{1/2}\|a\|_{2}.\]
Using that $k_{n}\lesssim (n+1)^{\gamma}$ completes the proof.
\end{proof}


\subsection{Examples}\label{RD filtrations examples}
We give some natural examples of filtrations which have either tracial polynomial growth or the rapid decay property. 

\begin{ex}\label{RDgroups}
 Suppose that $G$ is a group with a finite generating set $S$ with $S=S^{-1}$. Setting $B_{n}=(S\cup \{e\})^{n}$, we have a filtration of $C^{*}_{\lambda}(G)$ given by
\[
V_{n}=\Span\{\lambda(g):g\in B_{n}\}.
\]  
Let $\tau$ denote canonical trace on $C_\lambda^*(G)$.
It is direct to check that the rapid decay property  of $(C_\lambda^*(G),\tau)$ relative to
 $(V_{n})_{n}$ is equivalent to the rapid decay property of $G$ defined in \cite{Jolissaint}. We refer the reader to \cite{ChatterjiRDP} for a survey of results on the rapid decay property in the group context. It can also be readily verified, e.g. using Proposition \ref{polygrowthTFAE} (ii),
that $(C_\lambda^*(G),\tau)$ has tracial polynomial growth if and only if 
$n\mapsto \# B_n$ grows at most polynomially, i.e., 
if $G$ is a group with polynomial growth.
\end{ex}

\begin{ex}
Let $\mu$ be a Borel probability measure with compact support $K\subseteq \bR$. Consider, on the $C^*$-probability space $(C(K), \int \cdot \,d\mu)$, the filtration
\[
V_n=\{P\in \bC[t]:\deg(P)\leq n\}.
\]
Since $\dim V_n=n+1$ for all $n$, this filtration has rapid decay if an only if it has tracial polynomial growth (Proposition  \ref{polygrowthTFAE}). 
Let $P_0,P_1,P_2\ldots$
be the sequence of orthonormal polynomials associated to $\mu$, i.e., the orthonormal sequence obtained applying the Gram-Schmidt process to $\{1,t,t^2,\ldots\}$. By  
 Proposition  \ref{polygrowthTFAE},  $(V_n)_n$  has tracial polynomial growth 
 if and only if $\|P_{n}\|\leq c(n+1)^{\beta}$ for some constants $c,\beta\geq 0$ (here $\|P\|$ is the sup norm of $P$ over  $K$). 
Thus we have a connection between the rapid decay and polynomial growth properties with respect to the filtration by polynomials of degree at most $n$ and the theory of orthogonal polynomials. 
\end{ex}

\begin{ex}
Consider the case $A=C([-1,1])$ with $\tau(f)=\frac{1}{2}\int_{-1}^{1}f(x)\,dx$. Let $V_{n}=\{P\in \bC[t]:\deg(P)\leq n\}$. Let $P_{n}$ be the Legendre polynomials defined iteratively by $P_{0}(x)=1$, $P_{1}(x)=x$ and $(n+1)P_{n+1}(x)=(2n+1)xP_{n}(x)-nP_{n-1}(x)$ (see \cite[Section 4.5]{SzegoOrthoPolys}).
Then $P_{n}$ are orthogonal with respect to $\tau$ and $|P_{n}|\leq 1$ on $[-1,1]$ (see \cite[Theorem 7.3.21]{SzegoOrthoPolys}). Moreover $\|P_{n}\|_{2}=\frac{1}{\sqrt{2n+1}}$ (see \cite[equation 4.3.3]{SzegoOrthoPolys}). and $\Span\{P_{j}:0\leq j\leq n\}=V_{n}$. Thus setting $Q_{n}=\sqrt{2n+1}P_{n}$, we have that $\|Q_{n}\|_{C([-1,1])}\leq \sqrt{2n+1}$ and so the preceding example tells us that $(V_{n})_{n=0}^{\infty}$ has tracial polynomial growth, and the rapid decay property. Note that setting $\widetilde G_n(x)=G_n(2x-1)$
we obtain orthonormal polynomials relative to the Lebesgue measure in $[0,1]$. In this way we verify that $(V_n)_n$
has tracial polynomial growth in $C([0,1], \int \, dm)$.
\end{ex}

\begin{ex}\label{ex: any diffuse measure on interval}
Suppose that $a<b$ are real numbers and that $\mu$ is an atomless Borel probability measure whose support is the interval $[a,b]$. Then $(C([a,b]),\int\cdot \, d\mu)$
is isomorphic, as a $C^*$-probability space, to 
$(C([0,1]), \int\cdot \, dm)$ (where $m$ denotes the Lebesgue measure). Indeed, since $\mu$ is atomless and fully supported on $[a,b]$,  $F\colon [a,b]\to [0,1]$ given by $F(t)=\mu([a,t])$ is a homeomorphism such that $F_{*}(\mu)$ is the Lebesgue measure. 
Thus, we can always endow $C([a,b])$ with 
a filtration having tracial polynomial growth (and the rapid decay property) through its identification
with $(C([0,1]), \int\cdot \, dm)$ (see previous example).
\end{ex}

\begin{ex}
Consider the case $A=C([0,\pi])$ with $\tau(f)=\frac{1}{\pi}\int_{0}^{\pi}f(t)\,dt$. Let $h\in A$ be given by $h(x)=\cos(x)$. Set $V_{n}=\{P\circ h:P\in \bC[t], \deg(P)\leq n\}$. Since $h$ is injective on $[0,\pi]$, Stone-Weierstrass implies that this is indeed a filtration of $A$. Define $y_{k}(t)=\sqrt{2}\cos(t)$. We see iteratively using the product formula for cosine this $V_{n}=\Span\{y_{k}:k=0,\cdots,n\}$. Moreover, the $y_{k}$ are orthonormal with respect to $\tau$. Since $\|y_{k}\|=\sqrt{2}$, we see that $(V_{n})_n$ has tracial polynomial growth and the rapid decay property.
\end{ex}

\begin{ex}\label{ex:sc}
Consider the case $A=C([-2,2])$, with $\tau$ being integration against the semicircular distribution, i.e. $\tau(f)=\frac{1}{2\pi}\int_{-2}^{2}f(t)\sqrt{4-t^{2}}\,dt$. Let $U_{n}$ be the Tchebychev polynomials of the second kind defined iteratively by $U_{0}(x)=1$, $U_{1}(x)=2x$ and $U_{n+1}(x)=2xU_{n}(x)-U_{n-1}(x)$. 
We see by induction that $U_{n}(\cos(n\theta))=\frac{\sin(n+1)\theta}{\sin(\theta)}$ and that $U_{n}(\cos(\theta))=2\cos(\theta)+U_{n-2}(\cos(\theta))$. The last equation implies inductively that $\sup_{t\in [-1,1]}|U_{n}(t)|\leq n+1$. 
Set $G_{n}(t)=U_{n}(t/2)$. The relation $U_{n}(\cos(\theta))=1$ further implies that $(G_{n})_{n}$ and are orthonormal respect to $\tau$ (see \cite[Section 4.1]{SzegoOrthoPolys}). 
Further $\Span\{G_{k}:0\leq k\leq n\}=\{P\in \bC[t]:\deg(P)\leq n\}$. Since $\|G_{n}\|_{C([-2,2]}\leq n+1$, we see once again the filtration given by polynomials of degree at most $n$ have tracial polynomial growth and the rapid decay property with respect to the semicircular measure.  
\end{ex}

\begin{ex}
Let  $\mu$ be a Borel measure on $\bR$ with compact support $K$. The measure $\mu$
is called regular in the sense of Stahl and Totik, if it 
satisfies any of multiple equivalent conditions, one of which
is that
\[
\limsup_n |p_n(x)|^{\frac1n}\leq 1
\]
for quasi-every $x\in K$ (i.e., except on a set
of capacity zero). See \cite[Theorem 1.10]{simonequilibrium}. 
Here $(p_n)_{n=0}^\infty$ denotes the sequence of orthonormal polynomials with respect to $\mu$.
Thus,  if $\mu$ is non-regular and with compact support $K$ of positive capacity, there exists at least one $x\in K$ such that
$n\mapsto |p_n(x)|$ grows exponentially. In particular, 
$(C(K), \int \, d\mu)$ does not have rapid decay 
relative to the ``polynomial degree filtration" $\{P\in \bC[t]:\deg P\leq n\}$. See
\cite[Example 1.6]{simonequilibrium} for examples of non-regular measures with support $[-1,1]$.  By \cite[Corollary 5.6]{simonequilibrium}, there exist non-regular measures with support the standard Cantor set $K_C\subseteq [0,1]$ (whose capacity is positive). Thus, for such measures, $C(K_C)$ does not have rapid decay relative to the polynomial degree filtration. Note however, that since 
$C(K_C)$ is an AF $C^*$-algebra, $(C(K_C),\int\, d\mu)$ does admit filtrations with rapid decay, as remarked on in the next example.
\end{ex}
    
\begin{ex}\label{ex:AF}
Let $A$ be a separable AF $C^*$-algebra and $\rho$ a faithful state on $A$. Write $A=\overline{\bigcup_n A_n}^{\|\cdot\|}$, where $(A_n)_n$ is an increasing sequence of finite dimensional $C^*$-subalgebras. Let $c_{n}=\sup_{a\in A_{n}:\|a\|_{2}=1}\|a\|$, which is finite for every $n$ since $A_{n}$ is finite-dimensional and $\rho$ is faithful.
Choose any strictly increasing integer sequence $k_1<k_2<\ldots$ such that $c_n\leq k_n$, and define  $V_{k_n}=A_n$ for all $n$, and
$V_m=V_{k_n}$ for $k_n\leq m<k_{n+1}$. These spaces form a filtration of $A$ (since each $A_n$ is a $C^*$-subalgebra). Moreover, for $a\in V_m$ with $k_n\leq m<k_{n+1}$ we have
$\|a\|\leq c_n\|a\|_2\leq k_n\|a\|_2\leq m\|a\|_2$.
Thus $(A,\rho)$ has the rapid decay property with respect to $(V_{n})_{n}$. If $\rho$ is a trace, we further then have tracial polynomial growth.  

\end{ex}		
	
\begin{ex}
Let $G$ be a connected compact Lie group. Let $\tau$ be the state on $C(G)$ induced by the Haar measure.
Let us show here that $(C(G),\tau)$ has polynomial growth.
Given an irreducible unitary representation 
$\pi\colon G\to M_{n_\pi}(\bC)$
of $G$, define
\[
E_\pi=\mathrm{span}\{g\mapsto \pi(g)_{ij}:i,j=1,\ldots,n_\pi\}\subseteq C(G).
\]
That is, $E_\pi$ is the span of the ``matrix coefficients'' of $\pi$. Let $\widehat G$ denote the set of irreducible representations of $G$.  It is shown in \cite[Example 3.5]{Vergnioux}, and in more detail in \cite[Theorem 2.1]{BanicaVergnioux},
that a length function $\ell \colon \widehat G\to \{0,1,\ldots\}$ can be defined
such that the spaces 
\[
V_n=\mathrm{span}\{E_\pi:\ell(\pi)\leq n\}
\]
form a filtration of $C(G)$ with $n\mapsto \dim V_n$ having polynomial growth. We note that, by the Schur orthogonality relations, 
the functions
\[
x_{i,j,\pi}(g)=\sqrt{\dim n_\pi}\cdot \pi(g)_{ij},\quad i,j=1,\ldots,n_\pi,\quad \ell(\pi)\leq n, 
\]
form an o.n.b. of $V_n$. We readily verify via Proposition \ref{polygrowthTFAE} (ii) that  $(C(G),\tau)$ has tracial polynomial growth.
\end{ex}

\begin{ex}
Let $G$	be a connected compact Lie group and let $H$ be a closed subgroup. Identify $C(G/H)$ with the $C^*$-subalgebra of $C(G)$ of 
right $H$-invariant functions, i.e., such that $f(gh)=f(g)$
for all $g\in G$ and $h\in H$. Let us show that $(C(G/H),\tau|_{C(G/H)})$
has polynomial growth. For each $\pi\in \widehat G$, let 
$E_{\pi,H}=E_{\pi}\cap C(G/H)$.
It is known that $\bigcup_{\pi} E_{\pi,H}$ is norm dense in $C(G/H)$. (Indeed, the expectation
\[
f\mapsto P(f)(g):= \int_H f(gh)\, dh,\hbox{ for }g\in G,
\]
maps $E_{\pi}$ onto $E_{\pi,H}$ for all $\pi$, and thus  the norm dense set $\bigcup_{\pi} E_{\pi}$ onto $\bigcup_{\pi} E_{\pi,H}$). 
It follows that   
\[
V_{n,H}:=V_n\cap C(G/H)
\]
 is a filtration for $C(G/H)$. Since $V_{n,H}\subseteq V_n$ for all $n$, it follows by 
Proposition \ref{polygrowthTFAE} (iii) that $(C(G/H),\tau|_{C(G/H)})$ has tracial polynomial growth.
\end{ex}




\section{Rapid decay property for free products}\label{RD free product section}

\subsection{Preliminaries from Ricard-Xu}

A crucial part of our approach will be showing that if $(A_{j},\rho_{j}),j=1,2$ are two $C^{*}$-probability spaces with filtrations $(V_{n,j})_{n=0}^{\infty}$ with the rapid decay property, then $A_{1}*A_{2}$ has an appropriate filtration with the rapid decay property. In order to do this, we rely on fundamental results of Ricard-Xu \cite{KhinRX}, although we follow the presentation given in \cite{PisierRDP}.

Let $A_{j},j\in [m]$ be $C^{*}$-algebras, and for $j=1,2$ let $\rho_{j}$ be states on $A_{j}$ with faithful GNS representation. We set 
\begin{align*}
(A,\rho) &=(A_{1},\rho_{1})*(A_{2},\rho_{2})*\cdots (A_{m},\rho_{m}),\\
\cH &=L^{2}(A,\rho).
\end{align*}

For $\ell\in \bN$, we are interested in the following subspace of $A$:
\[W_{\ell}=\Span\{a_{1}\cdots a_{\ell}:a_{1}\cdots a_{\ell} \textnormal{ is an alternating centered word in $(A_{1},\cdots,A_{m})$}\}.\]
We define, by convention, $W_{0}=\bC 1$.

For  Hilbert spaces $\cH,\cK$, and $\xi\in \cK,\eta\in \cH$, we define $\omega_{\xi,\eta}\in B(\cH,\cK)$ by $\omega_{\xi,\eta}(\zeta)=\ip{\zeta,\eta}\xi$.
For an integer $0\leq r\leq \ell$, we define $s_{r}\colon W_{\ell}\to B(\cH)$ and $t_{r}\colon W_{\ell}\to B(\cH)\otimes_{\min}A$ to be the unique linear maps which satisfy:
\[s_{r}(a_{1}\cdots a_{\ell})=\omega_{a_{1}\cdots a_{r},a_{r+1}\cdots a_{\ell}},\]
\[t_{r}(a_{1}\cdots a_{\ell})=\omega_{a_{1}\cdots a_{r-1},a_{r+1}\cdots a_{\ell}}\otimes a_{r},\]
whenever $a_{1}\cdots a_{\ell}$ is an alternating centered word
in $(A_{1},\cdots,A_{m})$. Note that $W_{\ell}$ is isomorphic to the algebraic direct sum of $(A_{i_{1}}\ominus \bC 1)\otimes_{\textnormal{alg}}(A_{i_{2}}\ominus \bC 1)\otimes_{\textnormal{alg}}\cdots \otimes (A_{i_{\ell}}\ominus \bC 1)$ over all $i_{1},\cdots,i_{\ell}\in [m]$ with $i_{1}\ne i_{2},i_{2}\ne i_{3},\cdots$. Hence there are indeed unique linear maps on $W_{\ell}$ satisfying the above formulae.
For $x\in W_{\ell}$ we set
\[\kh(x)=\max_{0\leq r\leq \ell}(\max(\|s_{r}(x)\|,\|t_{r}(x)\|)).\]
The theorem of Ricard-Xu we need is the following.
\begin{thm}[Theorem 2.5 of \cite{KhinRX}]\label{thm:RX}
There are absolute constants $c,\beta>0$ such that
\[(c(\ell+1)^{\beta})^{-1} \kh(x)\leq \|x\|\leq 2(\ell+1)\kh(x)\]
for all $x\in W_{\ell}$.
\end{thm}
Our main purpose in this section is to establish a rapid decay property for free products. Thus we wish to bound the operator norm of an element in $W_{\ell}$ by its $2$-norm. For this, let us first note that $s_{r}$ induces an isometry between the closure of $W_{\ell}$ in $L^{2}(A,\rho)$ and $\HS(\cH)$. Moreover, the $W_{\ell}$ are orthogonal with respect to $\rho$ for distinct values of $\ell$ (this can be either be seen inductively using the defining alternating moment condition for free products, or realize that it follows automatically from Voiculescu's construction of the GNS Hilbert space for $\cH$). In particular,
\[\|s_{r}(x)\|\leq \|s_{r}(x)\|_{HS}=\|x\|_{2}\]
for all $x\in W_{\ell}$ and all $0\leq r\leq \ell$. 

\subsection{Proof of the rapid decay property for free products}

We use the following elementary generalization of the inequality between the operator and Hilbert-Schmidt norms.

\begin{lem}\label{lem: weak NC CS}
Let $\cH_{1},\cH_{2},\cK_{1},\cK_{2}$ be Hilbert spaces and let $(\xi_{i})_{i\in I},(\eta_{j})_{j\in J}$ be orthonormal systems in $\cK_{1},\cH_{1}$ respectively (not necessarily orthonormal bases). Let $F\subseteq I\times J$ be finite and suppose that for $(i,j)\in F$ we are given $T_{ij}\in B(\cH_{2},\cK_{2})$.  Then
\[\left\|\sum_{(i,j)\in F}\omega_{\xi_{i},\eta_{j}}\otimes T_{ij}\right\|\leq \left(\sum_{(i,j)\in F}\|T_{ij}\|^{2}\right)^{1/2}.\]

\end{lem}

\begin{proof}
Set $T=\sum_{(i,j)\in F}\omega_{\xi_{i},\eta_{j}}\otimes T_{ij}$. If $V_{1}$ is the closed linear span of $(\eta_{j})_{j\in J}$, then it is direct check that $T$ vanishes on $V_{1}^{\perp}\otimes \cH_{2}$. Hence we may, and will, assume that $(\eta_{j})_{j\in J}$ is an orthonormal basis for $\cH_{1}$. Given $v\in \cH_{1}\otimes \cH_{2}$, we may thus write $v=\sum_{k}\eta_{k}\otimes v_{k}$ with $v_{j}\in \cH_{2}$. Then:

\begin{align}
\|Tv\|=\left\|\sum_{(i,j)\in F}\xi_{i}\otimes T_{ij}v_{j}\right\| &=\sum_{i}\left\|\sum_{j:(i,j)\in F}T_{ij}v_{j}\right\|^{2}\\
&\leq\sum_{i}\left(\sum_{j:(i,j)\in F}\|T_{ij}\|\|v_{j}\|\right)^{2}\leq \|v\|^{2}\sum_{(i,j)\in F}\|T_{ij}\|^{2},
\end{align}
where in the last step we use Cauchy-Schwarz, plus the fact that $\|v\|^{2}=\sum_{k}\|v_{k}\|^{2}$.
\end{proof}

\begin{lem}\label{lem: defining the filtration for RDP}
Suppose we are given $m\in \bN$ and  $C^{*}$-probability spaces $(A_{j},\rho_{j}),j\in [m]$. Let $(V_{n,j})_{n=0}^{\infty}$ be filtrations of $A_{j}$, $j\in [m]$. Let $V_{k}\subseteq A_{1}*A_{2}*\cdots*A_{m}$ be the span of all alternating centered words in $(V_{1},\cdots,V_{m})$ of length at most $k$. 
Then 
\begin{enumerate}[(i)]
    \item $(V_{n})_{n=0}^{\infty}$ is a filtration of $A_{1}*A_{2}*\cdots*A_{m}$, \label{item: free filtration}
    \item $V_{n}$ contains the span of all words of the form $a_{1}\cdots a_{\ell}$ such that  there is a function $k\colon [\ell]\to \bN$ with $\sum_{i}k(i)\leq n$ and a function $j\colon [\ell]\to [m]$ (not assumed to be alternating) with $a_{i}\in V_{k(i),j(i)}$. \label{item: free degree terms}
\end{enumerate}
\end{lem}

\begin{proof}

The fact that $V_{n}$ is $*$-closed is direct, so we focus on showing that $V_{n}V_{k}\subseteq V_{n+k}$.

Note that (\ref{item: free degree terms}) follows once we show that $V_{k}V_{n}\subseteq V_{n+k}$ for all $n,k\in \bN\cup\{0\}$. 
 Indeed, if $B$ is the $*$-algebra in $A_{1}*A_{2}*\cdots*A_{m}$ generated by $\bigcup_{j=1}^{m}\bigcup_{n}V_{n,j}$ then $B$ is the span of all words of the form $a_{1} \cdots a_{\ell}$ with $a_{i}\in V_{k(i),j(i)}$ for some functions $k\colon [\ell]\to \bN$ and $j\colon [\ell]\to [m]$. Setting $K=\sum_{i}k(i)$, we have that $a_{1}\cdots a_{\ell}\in V_{K}$ if $V_{k}V_{n}\subseteq V_{n+k}$ for all $n,k\in \bN\cup\{0\}$.   
 Note further that (\ref{item: free degree terms}) also implies that $\bigcup_{n}V_{n}$ is dense in $A$.
 
 So it suffices to show that $V_{n}V_{k}\subseteq V_{n+k}$ for all $n,k\in \bN\cup\{0\}$.
 To prove this, for $\ell,k\in \bN\cup\{0\},$ let $\Omega_{\ell,k}$ be the set of all alternating products in $(V_{k,1}\ominus \bC1,V_{k,2}\ominus \bC 1)$ of length $\ell$. Set $W_{\ell,k}=\sum_{p=0}^{\ell}\Span(\Omega_{\ell,k})$, so that $V_{k}=W_{k,k}$. It then suffices to show that $W_{\ell,k}W_{r,n}\subseteq W_{\ell+r,k+n}$ for all integers $\ell,r,k,n\geq 0$. By linearity, it suffices to show that $\Omega_{\ell,k}\Omega_{r,n}\subseteq W_{\ell+k,r+n}$ for all integers $\ell,k,r,n\geq 0.$

We prove by induction on $\ell\geq 0$ that for all $r,k,n\in \bN\cup\{0\}$ we have $\Omega_{\ell,k}\Omega_{r,n}\subseteq W_{\ell+r,k+n}$. The base case $\ell=0$ is direct.

For the inductive step assume that $\ell\geq 1$ and the desired result is true for all $0\leq p\leq \ell-1$. Note that $\Omega_{\ell,k}\Omega_{0,n}=\Omega_{\ell,k}$ for all $n\geq 1$, and that $\Omega_{\ell,k}\Omega_{r,0}=\{0\}$ for all $r\geq 0$. Similarly $\Omega_{\ell,0}\Omega_{r,n}=\{0\}$ for all $r,n\geq 0$. So it suffices to show that if $r,k,n\in \bN$, then $\Omega_{\ell,k}\Omega_{r,n}\subseteq W_{\ell+r,k+n}$. Fix alternating functions $j\colon [\ell]\to[m]$ and $f\colon [r]\to [m]$, and let $a_{i}\in V_{k,j(i)}\ominus \bC 1$, $b_{s}\in V_{k,f(s)}\ominus \bC 1$ are given. If $j(\ell)\ne f(1)$, then $a_{1}\cdots a_{\ell}b_{1}\cdots b_{r}$ is an alternating product in $\Omega_{\ell+r,\max(k,n)}\subseteq \Omega_{\ell+r,k+n}$. 

If $j(\ell)=f(1)$, let $x=a_{\ell}b_{1}-\rho_{j(\ell)}(a_{\ell}b_{1})\in V_{k+n,j(\ell)}\ominus \bC 1$. Set $\widetilde{a}=a_{1}\cdots a_{\ell-1}$ and $\widetilde{b}=b_{2}\cdots b_{r-1}$. Then:
\[a_{1}\cdots a_{\ell}b_{1}\cdots b_{r}=\rho_{j(\ell)}(a_{\ell}b_{1})\widetilde{a}\widetilde{b}+\widetilde{a}x\widetilde{b}.\]
Since $\widetilde{a}\in \Omega_{\ell-1,k},\widetilde{b}\in \Omega_{r-1,n}$ and $r\geq 1$, the first term is in $W_{\ell+r-2,k+n}\subseteq W_{\ell,k+n}$, by inductive hypothesis. 
The second term is an alternating product which is in $\Omega_{\ell+r,k+n}$. Thus $a_{1}\cdots a_{\ell}b_{1}\cdots b_{r}\in W_{\ell+r,k+n}$.

\end{proof}

\begin{lem}\label{lem: RDP estimate on layers for free prod}
Suppose that $(A_{j},\rho_{j})$, $j\in [m]$ are  $C^{*}$-probability spaces. 
Suppose we are given linear subspaces $E_{j}\subseteq A_{j}\ominus \bC 1$ for $j\in [m]$. 
For $\ell\in \bN$ set
\[W_{\ell}=\Span\{a_{1}\cdots a_{\ell}:a_{i}\in E_{j(j)} \textnormal{ where $j\colon [\ell]\to [m]$ is alternating}\}.\]
Assume that 
\[\|x\|\leq C_{j}\|x\|_{2}\]
for all $x\in E_{j}$.
Then for $y\in W_{\ell}$ we have
\[\|y\|\leq 2\sqrt{m} (\ell+1)\left(\max_{j}C_{j}\right)\|y\|_{2}.\]
\end{lem}

\begin{proof}
By Lemma \ref{lem: passing to HS closure} we may, and will, assume that $E_{j}$ is complete under $\|\cdot\|_{2}$. 
We first set up some notation.
Let $(e_{i}^{(j)})_{i\in I_{j}}$ for $j=1,2$ be orthonormal bases for $E_{j}$. 
Given an integer $p\in \bN$, we let
\[\cB_{p}=\{(i_{1},\cdots,i_{p}):i_{s}\in I_{j(s)} \textnormal{ and } j\colon [p]\to [m] \textnormal{ is alternating}\}.\]
For $I=(i_{1},\cdots,i_{p})\in \cB_{p}$, with $i_{s}\in I_{j(s)}$, let $e_{I}=e_{i_{1}}^{(j(1))}e_{i_{2}}^{(j(2))}\cdots e_{i_{l}}^{(j(\ell))}$.  For $0\leq r\leq \ell$ and $I\in \cB_{\ell}$, we set $I_{<r}=(i_{1},\cdots,i_{r-1}),I_{>r}=(i_{r+1},\cdots,i_{l})$.

Let $y\in W_{\ell}$, and write $y=\sum_{\widehat{I}\in \cB_{l}}\lambda_{\widehat{I}}e_{\widehat{I}}$.  For $I=(i_{1},\cdots,i_{r-1})\in \cB_{r-1}$ with $i_{r-1}\in I_{j(r-1)}$  and $\widetilde{I}=(\widetilde{i}_{1},\cdots,\widetilde{i}_{\ell-r})\in \cB_{\ell-r}$ with $\widetilde{i}_{1}\in I_{\widetilde{j}(1)}$, set 
\[T_{I,\widetilde{I}}=\sum_{j\in [m]\setminus \{j(r-1),\widetilde{j}(1)\}}\sum_{i\in I_{j}}\lambda_{I,i,\widetilde{I}}e_{i}.\] 
Then:
\[t_{r}(y)=\sum_{(I,\widetilde{I})\in \cB_{r-1}\times \cB_{\ell-r}}\omega_{e_{I},e_{\widetilde{I}}}\otimes T_{I,\widetilde{I}}.\]
The elements $e_{I},e_{\widetilde{I}}$ may not be orthogonal to each other. However, it is still true that both systems $(e_{I})_{I\in \cB_{r-1}}$, $(e_{\widetilde{I}})_{\widetilde{I}\in \cB_{l-r}}$ are individually orthonormal systems. We can thus apply Lemma \ref{lem: weak NC CS} to see that
\begin{equation}\label{eqn:CS number 1}
\|t_{r}(y)\|^{2}\leq \sum_{(I,\widetilde{I})\in \cB_{r-1}\times \cB_{l-r}}\|T_{I,\widetilde{I}}\|^{2}.
\end{equation}
By the assumed $L^{2}-L^{\infty}$ estimates:
\begin{align*}
\|T_{I,\widetilde{I}}\|&\leq \sum_{j\in [m]\setminus\{j(r-1),\widetilde{j}(1)}\left\|\sum_{i\in I_{j}}\lambda_{I.i,\widetilde{I}} e_{i}\right\|\\
&\leq \left(\max_{j} C_{j}\right)\sum_{j\in [m]\setminus\{j(r-1),\widetilde{j}(1)}\left(\sum_{i\in I_{j}}|\lambda_{I.i,\widetilde{I}}|^{2}\right)^{1/2}\\
&\leq  \sqrt{m}\left(\max_{j} C_{j}\right)\left(\sum_{j\in [m]\setminus \{j(r-1),\widetilde{j}(1)\}}\sum_{i\in I_{j}}|\lambda_{I,i,\widetilde{I}}|^{2}\right)^{1/2}.
\end{align*}
Thus 
\begin{align}\label{eqn:CS number 2}
\left(\sum_{(I,\widetilde{I})\in \cB_{r-1}\times \cB_{\ell-r}}\|T_{I,\widetilde{I}}\|^{2}\right)^{1/2}
&\leq\sqrt{m}\left(\max_{j}C_{j}\right) \left(\sum_{I\in \cB_{\ell}}|\lambda_{\widehat{I}}|^{2}\right)^{1/2}\\
&=\sqrt{m}\left(\max_{j}C_{j}\right)\|y\|_{2}.\nonumber
\end{align}

Plugging (\ref{eqn:CS number 1}) into (\ref{eqn:CS number 2}) we see that:
\[\|t_{r}(y)\|\leq \sqrt{m}\left(\max_{j}C_{j}\right)\|y\|_{2}.\]
As note before, we have that 
\[\|s_{r}(y)\|\leq \|s_{r}(y)\|_{HS}=\|y\|_{2}.\]
Since the above two estimates hold for all $r$, we may apply Theorem \ref{thm:RX} to see that 
\[
\|y\|\leq 2(\ell+1)\kh(a)\leq 2(\ell+1)\left(\max_{j}C_{j}\right)\|y\|_{2}.
\]
\end{proof}

\begin{thm}\label{thm: RDP for free products}
Let $(A_{j},\rho_{j}),j\in [m]$ be  unital $C^{*}$-probability spaces and let $(V_{n,j})_{n=0}^{\infty}$ be filtrations of $A_{j}$, $j\in [m]$. Let $(V_{n})_{n=0}^{\infty}$ be the filtration of $(A_{1},\rho_{1})*\cdots*(A_{m},\rho_{m})$ defined in Lemma \ref{lem: defining the filtration for RDP}. If each $(V_{n,j})_{n=0}^{\infty}$ has the rapid decay property for $j\in [m]$, then so does $(V_{n})_{n=0}^{\infty}$. 
    
\end{thm}

\begin{proof}
For $\ell,p\in \bN$, let
\[W_{\ell,p}=\Span\{a_{1}\cdots a_{\ell}:a_{i}\in V_{p,j(i)}\ominus \bC 1, \textnormal{ where $j\colon [\ell]\to [m]$ is alternating} \}.\]
Then, by definition,
\[V_{n}=\sum_{\ell=0}^{n}W_{\ell,n},\]
and the construction of the free product Hilbert space tells us the vector spaces in this sum are pairwise orthogonal.  Given $x\in V_{n}$, write $x=\sum_{\ell=0}^{n}x_{\ell}$ with $x_{\ell}\in W_{\ell,n}$. Then, by Lemma \ref{lem: RDP estimate on layers for free prod}, it follows that 
\begin{align*}
  \|x\|\leq \sum_{\ell=0}^{n}\|x_{\ell}\|&\lesssim \sqrt{m}\sum_{\ell=0}^{n}(\ell+1)(n+1)^{\max(\alpha_{1},\alpha_{2})}\|x_{\ell}\|_{2}\\
  &\leq \sqrt{m}(n+1)^{\max(\alpha_{1},\alpha_{2})}\|x\|_{2}\left(\sum_{\ell=0}^{n}(\ell+1)^{2}\right)^{1/2}\\
  &\leq \sqrt{m}(n+1)^{3/2+\max(\alpha_{1},\alpha_{2})}\|x\|_{2},  
\end{align*}
where in the second-to-last step we use Cauchy-Schwarz and the fact that the terms in the sum $a=\sum_{\ell=0}^{n}a_{\ell}$ are pairwise orthogonal. 
    
\end{proof}

\begin{cor}
Let $(A_{j},\rho_{j})$, $j=1,2$ be unital $C^{*}$-probability spaces, and let $B_{j}\subseteq A_{j}$ be finitely generated tracial $*$-algebra. If each $B_{1},B_{2}$ has the rapid decay property, so does $B_{1}*B_{2}$. In particular, if $x_{j}\in A_{j}^{r_{j}}$ are tuples with the rapid decay property, then $(x_{1},x_{2})$ has the rapid decay property.
\end{cor}

\begin{proof}
Let $x_{j}\in B_{j}^{r_{j}}$, $j=1,2$ be tuples which generate $B_{j}$. Set \[V_{n,j}=\{P(x_{j}):P\in \bC^{*}\ip{T_{1},\cdots,T_{r_{j}}} \textnormal{ and $\deg(P)\leq n$}\}.\] 
Let $(V_{n})_{n=0}^{\infty}$ be as in Lemma \ref{lem: defining the filtration for RDP}. Then, by Lemma \ref{lem: defining the filtration for RDP} (\ref{item: free degree terms}) it follows that  $V_{n}$ contains \[\{P(x_{1},x_{2}):P\in \bC^{*}\ip{T_{1},\cdots,T_{r_{1}+r_{r}}} \textnormal{ and $\deg(P)\leq n$}\}.\]
Hence, by Theorem \ref{thm: RDP for free products}, if $P\in\bC^{*}\ip{T_{1},\cdots,T_{r_{1}+r_{r}}}$ has degree at most $n$, then 
\[\|P(x_{1},x_{2})\|\lesssim (n+1)^{\beta}\|P(x_{1},x_{2})\|_{2}\]
for appropriate $\beta>0$. 
\end{proof}

We have shown that reduced free products of $C^*$-algebras with rapid decay  again have rapid decay. 
In the following sections we 
exploit this fact to investigate when reduced free product $C^*$-algebras are selfless.

\section{Selflessness via an Avitzour condition}\label{Avitzour section}

Recall that if $(A,\rho)$ is a $C^{*}$-probability space, then the \emph{centralizer of $\rho$} is, by definition,
\[A^{\rho}=\{x\in A:\rho(xy)=\rho(yx) \textnormal{ for all $y\in A$}\}.\]
The goal of this section is to prove the following, we note the similarity to \cite[Section 3]{avitzour1982free}.\

\begin{thm}\label{thm: the Avitzour approach}
Let $(A_{j},\rho_{j})$, for $j=1,2$,  be $C^{*}$-probability spaces with filtrations $(V_{n,j})_{n=0}^{\infty}$ with the rapid decay property. Suppose there exist unitaries $u\in \bigcup_{n}V_{n,1}$, and $v,w\in \bigcup_{n}V_{n,2}$, such that $v$ lies in  $A_{2}^{\rho_{2}}$, and 
\[
\rho_1(u)=\rho_2(v)=\rho_2(w)=\rho_{2}(v^{*}w)=0. 
\]
Then $(A_{1},\rho)*(A_{2},\rho_{2})$ is selfless.
\end{thm}

Note that finite-dimensional $C^{*}$-probability spaces $(A,\rho)$ have the rapid decay property with the filtration $V_{n}=A$ for all $n$.  
We can thus deduce from Theorem \ref{thm: the Avitzour approach} that  
$(\mathbb{M}_m(\mathbb{C}),\rho_1)*(\mathbb{M}_n(\mathbb{C}),\mathrm{tr})$ is selfless once $m,n\geq 2$ and
$\rho_{1}$ is faithful. If $\rho_1$ is additionally nontracial, then  $(\mathbb{M}_m(\mathbb{C}),\rho_1)*(\mathbb{M}_n(\mathbb{C}),\mathrm{tr})$ is purely infinite. See Corollary \ref{MmstarMn} below.

The main preliminary result we need is the following, which builds off of a method of Avitzour.

\begin{lem}\label{lem: Avitzour argument}
Let $(A_{j},\rho_{j}),j=1,2$ be $C^{*}$-probability spaces, $u\in \cU(A_{1})$, $v,w\in \cU(A_{2})$ with $v\in A_{2}^{\rho_{2}}$ and $\rho_{2}(v^{*}w)=0=\rho_{2}(v)=\rho_{2}(w)=\rho_{1}(u)$. Let $a=a_{1}\cdots a_{\ell}$ be an alternating word in $(A_{1}\ominus \bC 1,A_{2}\ominus \bC 1)$. Assume that $n>\frac{\ell}{2}>0$. Then:

\begin{enumerate}[(i)]
    \item $(wuw)(uv)^{n}a$ is  a linear combination of alternating centered words in $(A_{1},A_{2})$ of positive length which start with $w$, \label{item: move to the front}
    \item $a(v^{*}u^{*})^{n}(w^{*}uw^{*})$ is a linear combination of alternating centered words in $(A_{1},A_{2})$ positive length which end with $w^{*}$.\label{item: move to the back}
    \item $(wuw)(uv)^{n}a(v^{*}u^{*})^{n}(w^{*}u^{*}w^{*})$ is a linear combination of alternating centered words in $(A_{1},A_{2})$ of positive length which start with $w$ and end with $w^{*}$. \label{item: move all around now}
\end{enumerate}

\end{lem}

\begin{proof}

(\ref{item: move to the front}): It suffices to show that once $n>\frac{\ell}{2}$, then $(uv)^{n}a$ is a linear combination of alternating centered words in $(A_{1},A_{2})$ of positive length which start with $u$. 

We first separately handle the cases $\ell=1,2$. If $\ell=1$, then $(uv)^{n}a=(uv)^{n-1}uva_{1}$. If $a_{1}\in A_{1}$, then this is already an alternating, centered word. If $a_{1}\in A_{2}$, set $\widetilde{a}=va_{1}-\rho_{2}(va_{1})$. Then:
\[(uv)^{n}a=(uv)^{n-1}u\widetilde{a}+\rho_{2}(va_{1})(uv)^{n-1}u.\]
Both terms are then alternating centered words which start with $u$.
If $\ell=2$, then necessarily $n\geq 2$ and $(uv)^{n}a=(uv)^{n-1}uva_{1}a_{2}$. If $a_{1}\in A_{1}$, then this is again already an alternating, centered word. Otherwise, $a_{1}\in A_{2}$, so $a_{2}\in A_{1}$. Set $\widetilde{a}_{1}=va_{1}-\rho_{2}(va_{1}),$ and $\widetilde{a}_{2}=ua_{2}-\rho_{1}(ua_{1})$. Then:
\[(uv)^{n}a_{1}a_{2}=(uv)^{n-1}u\widetilde{a}_{1}a_{2}+\rho_{2}(va_{1})(uv)^{n-1}u\widetilde{a}_{2}+\rho_{1}(ua_{2})\rho_{2}(va_{1})(uv)^{n-1}.\]
Since $n\geq 2$, this is a linear combination of alternating, centered words which start with $u$.

We now prove by induction on $n$ that for all $0<\ell<2n$, any alternating centered word $a=a_{1}\cdots a_{\ell}$ in $(A_{1},A_{2})$  we have that $(uv)^{n}a$ is a linear combination of alternating centered words in $(A_{1},A_{2})$ starting with $u$. The case $\ell=1$ above handles the base case $n=1$. So we assume that $n\geq 2$, that the result is true for all $1\leq k<n$, and try to prove the result is true for $n$.  By the above paragraph, it suffices to show that if $2n>\ell\geq 3$, then  for any alternating centered word $a=a_{1}\cdots a_{\ell}$ in $(A_{1},A_{2})$ we have that $(uv)^{n}a$ is a linear combination of alternating centered words starting with $u$. 
If $a_{1}\in A_{1}$, then $(uv)^{n}a$ is already an alternating centered word. Assume that $a_{1}\in A_{2},$ so that $a_{2}\in A_{1}$. Set $\widetilde{a}_{1}=va_{1}-\rho_{2}(va_{1})$, $\widetilde{a}_{2}=ua_{2}-\rho_{1}(ua_{2})$. Then:
\begin{align*}
(uv)^{n}a&=(uv)^{n-1}u\widetilde{a}_{1}a_{2}\cdots a_{\ell}+\rho_{2}(va_{1})(uv)^{n-1}\widetilde{a}_{2}a_{3}\cdots a_{\ell}\\
&+\rho_{2}(va_{1})\rho_{1}(ua_{2})(uv)^{n-1}a_{3}\cdots a_{\ell}.
\end{align*}
The first two terms are alternating centered words starting in $u$. Since $\ell\geq 3$, we have that $\ell-2\geq 1>0$ and since $\frac{\ell}{2}<n$, we have $\frac{\ell-2}{2}<n-1$. Thus by inductive hypothesis, we have that $(uv)^{n-1}a_{3}\cdots a_{\ell}$ is a linear combination of alternating centered words with start with $u$. This completes the proof of the inductive step.

(\ref{item: move to the back}): This can be argued in a manner similar to (\ref{item: move to the front}), alternatively we can take adjoints to deduce (\ref{item: move to the back}) from (\ref{item: move to the front}).

(\ref{item: move all around now}): It is explicitly noted in \cite[Proposition 3.1]{avitzour1982free} that $(uv)^{n}a(v^{*}u^{*})^{n}$ is a linear combination of words of the form:
\begin{itemize}
    \item a word of positive length starting in $u$ and ending in $u^{*}$,
    \item $(uv)^{j}$ with $j\in \bZ\setminus\{0\}$,
\end{itemize}
(alternatively this is a direct induction argument along the lines of (\ref{item: move to the back})). We remark that it is already important in the case $n=1$ that $v\in A_{2}^{\rho_{2}}$ so that $v^{*}av\in A_{2}\ominus \bC 1$ if $a\in A_{2}\ominus \bC 1$. 
The above bullet points show that $(wuw)(uv)^{n}a(v^{*}u^{*})^{n}(w^{*}u^{*}w^{*})$ is of the desired form. Here let us explicitly note that if $j>0$, then \[(wuw)(uv)^{j}(w^{*}u^{*}w^{*})\]can be regarded as an alternating centered word starting in $w$ and ending in $w^{*}$, since $vw^{*}\in A_{2}\ominus \bC 1$. Similar remarks apply to 
\[
(wuw)(uv)^{-j}(wuw)=(wuw)(v^{*}u^{*})^{j}(wuw),
\]
since $wv^{*}\in A_{2}\ominus \bC 1$.
\end{proof}

We use the previous lemma to prove the following asymptotic $L^{2}$-$L^{2}$ isometry property, which in turn will be used to mimic the proof strategy in \cite{sri2025strictcomparisonreducedgroup}.

\begin{lem}\label{item: asy L2 L2 isometry avitzour}
Let $(A_{j},\rho_{j}),j=1,2$ be $C^{*}$-probability spaces, suppose that $u\in \cU(A_{1}),v,w\in \cU(A_{2})$  with $v\in A_{2}^{\rho_{2}}$ and $\rho_{2}(v^{*}w)=0=\rho_{2}(v)=\rho_{2}(w)=\rho_{1}(u)$.

Let $B_{1}=A_{1}$, $B_{2}=A_{2}$, and $B_3=A_1$. Consider the unique $*$-homomorphism $\phi_{n}\colon B_{1}*_{\textnormal{alg}}B_{2}*_{\textnormal{alg}}B_{3}\to A_{1}*A_{2}$ which is the identity on $B_{1},B_{2}$ and is conjugation by $x_{n}^{*}$ where $x_{n}=(wuw)(uv)^{n}$. Fix $\ell\in \bN\cup\{0\}$. Then:

\begin{enumerate}[(i)]
    \item \label{item: asy trace-preserving avitzour} for any $a\in B_{1}*_{\textnormal{alg}}B_{2}*_{\textnormal{alg}}B_{3}$ which is a linear combination of alternating, centered words in $(B_{1},B_{2},B_{3})$ of length at most $\ell$, we have 
    \[(\rho_{1}*\rho_{2})(\phi_{n}(a))= (\rho_{1}*\rho_{2}*\rho_{1})(a),\]
    for all $n>\frac{\ell}{2}$.
    \item \label{item: asy L2 isom avitzour} for every $a\in B_{1}*_{\textnormal{alg}}B_{2}*_{\textnormal{alg}}B_{3}$ which is a linear combination of alternating centered words in $(B_{1},B_{2},B_{3})$ of length at most $\ell$, we have that 
    \[\|\phi_{n}(a)\|_{2}=\|a\|_{2},\]
    whenever $n>\ell$. 
\end{enumerate}

\end{lem}

\begin{proof}

(\ref{item: asy trace-preserving avitzour}): By linearity, it suffices to show the desired claim when $a$ is an alternating centered word in $(B_{1},B_{2},B_{3})$. Since $\phi_{n}$ is unital, the claim holds if $a$ is a word of length zero. So we assume that $a$ is an alternating centered word of length $\ell>0$ with $\ell<2n$. Write 
\[
a=y_{0}\iota_{3}(c_{1})y_{1}\iota_{3}(c_{2})\cdots y_{t-1}\iota_{3}(c_{t})y_{t},
\]
where :
\begin{itemize}
    \item $c_{i}\in A_{1}\ominus \bC 1$, and
    \item $y_{i}$ is an alternating centered word in $(B_{1},B_{2})$
    \item $y_{i}$ has positive length if $1\leq i\leq t-1$,
    \item $y_{0},y_{t}$ are potentially empty. 
\end{itemize} 
Then:
\begin{align*}
  \phi_{n}(a)&=y_{0}x_{n}^{*}c_{1}x_{n}y_{1}x_{n}^{*}c_{2}x_{n}\cdots y_{t-1}x_{n}^{*}c_{t}x_{n}y_{t}\\
  &=(y_{0}x_{n}^{*})c_{1}(x_{n}y_{1}x_{n}^{*})c_{2}\cdots (x_{n}y_{t-1}x_{n}^{*})c_{t}(x_{n}y_{t}).
\end{align*}
Note that each $y_{i}$ has length at most $\ell$. Thus if $n>\frac{\ell}{2}$, Lemma \ref{lem: Avitzour argument} implies that $x_{n}y_{i}x_{n}^{*}$ starts with $w$ and ends in $w^{*}$ if $1\leq i\leq t-1$, and that $(y_{0}x_{n}^{*})$ ends in $w^{*}$ and $(x_{n}y_{t})$ starts with $w$ (note that these last two claims are particularly straightforward if $y_{0}$ or $y_{t}$ are empty). 
Thus once $n>\frac{\ell}{2}$, we have that $\phi_{n}(a)$ is an alternating centered word of positive length and thus $(\rho_{1}*\rho_{2})(\phi_{n}(a))=0$. 

(\ref{item: asy L2 isom avitzour}): 
By Lemma \ref{lem: defining the filtration for RDP}, we know that $a^{*}a$ is a linear combination of alternating, centered words of length at most $2\ell$. Thus this follows from (\ref{item: asy trace-preserving avitzour}) and the fact that $\phi_{n}$ is a $*$-homomorphism.

\end{proof}

\begin{proof}[Proof of Theorem \ref{thm: the Avitzour approach}]
Adopt notation as in Lemma \ref{item: asy L2 L2 isometry avitzour}. We equip $B_{1}*B_{2}*B_{3}$ with the filtration $(V_{n})_{n=0}^{\infty}$ coming from Lemma \ref{lem: defining the filtration for RDP} and the filtrations $(\iota_{j}(V_{n,s}))_{n=0}^{\infty}$ when $j\in [3],s\in [2]$ have the same parity. Equip $A_{1}*A_{2}$ with the filtration $(E_{n})_{n=0}^{\infty}$ coming from Lemma \ref{lem: defining the filtration for RDP} and the filtrations $(V_{n,1})_{n},(V_{n,2})_{n}$. Fix $k\in \bN$ such that $u\in V_{k,1}$, $v,w\in V_{k,2}$. By Theorem \ref{thm: RDP for free products}, we may find $c,\alpha>0$ such that
\[\|x\|\leq C(n+1)^{\alpha}\|x\|_{2}, \textnormal{ for all $x\in E_{n}$.}\]

Fix $m\in \bN$  with $m>k$ and fix $n>m$. We wish to understand how far in the filtration $(E_{s})_{s}$ the subspace $\phi_{n}(V_{n})$ is.
Suppose that $a\in B_{1}*_{\textnormal{alg}}B_{2}*_{\textnormal{alg}}B_{3}$ is an alternating centered word in $(\iota_{1}(V_{n,1}),\iota_{2}(V_{n,2}),\iota_{3}(V_{n,1}))$ of length $\ell>0$ with $\ell\leq m$. As in the proof of Lemma \ref{item: asy L2 L2 isometry avitzour}, write 
\[a=y_{0}\iota_{3}(c_{1})y_{1}\iota_{3}(c_{2})\cdots y_{t-1}\iota_{3}(c_{t})y_{t},\]
where $y_{i}$ are alternating centered words in $(V_{n,1},V_{n,2}),$ and $c_{i}\in V_{n,1}\ominus \bC 1$, with $y_{0},y_{t}$ potentially being empty and $y_{i}$ has positive length for $1\leq i\leq t-1$. Then
\[\phi_{n}(a)=y_{0}(x_{n}^{*}c_{1}x_{n})y_{1}\cdots (x_{n}^{*}c_{t}x_{n})y_{t}.\]
Each $y_{i}\in E_{m}$, and each $x_{n}^{*}c_{i}x_{n}\in E_{5n+6}$. Repeatedly using that $E_{k_{1}}E_{k_{2}}\subseteq E_{k_{1}+k_{2}}$, we have that $\phi_{n}(a)\in E_{(2t+1)(5n+6)}\subseteq E_{(2m+1)(5n+6)}$. 
Thus for all $n>m$,  we have $\phi_{n}(V_{m})\subseteq E_{(2m+1)(5n+6)}$, and in particular
\[\|\phi_{n}(a)\|\leq c((2m+1)(5n+6)+1)^{\alpha}\|\phi_{n}(a)\|_{2}.\]
Moreover, by Lemma \ref{item: asy L2 L2 isometry avitzour}, we have that  $\|\phi_{n}(a)\|_{2}=\|a\|_{2}$, since $n>m$. 

Fix $x\in V_{m}$. Suppose that $n>2m$ and pick $r\in \bN$ with $r\geq 2$ such that $2mr<n\leq 2m(r+1)$. Then:
\begin{align*}
\|\phi_{n}(x)\|^{2r}=\|\phi_{n}((x^{*}x)^{r})\|&\leq c((4mr+1)(5n+6)+1))^{\alpha}\|(x^{*}x)^{r}\|_{2}\\
&\leq  c((4mr+1)(10m(r+1)+6)+1))^{\alpha}\|x\|^{2r}.    
\end{align*}
Letting $n\to\infty$ we have
\begin{align*}
\limsup_{n\to\infty}\|\phi_{n}(x)\|&\leq \limsup_{r\to\infty}c^{1/2r}((4m+1)(10m(r+1)+6)+1))^{\alpha/2r}\|x\|=\|x\|.
\end{align*}
So we have shown that 
\[\limsup_{n\to\infty}\|\phi_{n}(x)\|\leq \|x\|\]
 for all $x\in \bigcup_{m}V_{m}$. Moreover, by Lemma \ref{item: asy L2 L2 isometry avitzour}, we have that $(\rho_{1}*\rho_{2})(\phi_{n}(x))\to (\rho_{1}*\rho_{2}*\rho_{1})(x)$ for all $x\in \bigcup_{m}V_{m}$.  Since $\bigcup_{m}V_{m}$ is a dense $*$-subalgebra of $A_{1}*A_{2}$, we conclude the proof by applying Lemma \ref{wk convergence+ strong semi conv}.
\end{proof}

\begin{cor}\label{MmstarMn}
Let
\[
(A,\rho)\cong (\mathbb{M}_m(\mathbb{C}),\rho_1)*(\mathbb{M}_n(\mathbb{C}),\mathrm{tr}),
\]
where $\rho_1$ is faithful and $m,n\geq 2$. Then $(A,\rho)$
is selfless. Consequently, if $\rho_1$ is a trace,
then $A$ has strict comparison and stable rank one, and $A$ is purely infinite otherwise.
\end{cor}
\begin{proof}
Clearly $\mathbb{M}_m(\mathbb{C})$ and $\mathbb{M}_n(\mathbb{C})$ have rapid decay relative to constant filtrations (see Example \ref{ex:AF}). It remains to find unitaries $u,v,w$ as in Theorem \ref{thm: the Avitzour approach}. Write $\rho_1(\cdot)=\mathrm{Tr}((\cdot)a)$, and assume without loss of generality that $a$
is diagonal. Now choose $u$ and $w$ to be permutation unitaries
with zeros along the main diagonal, thus fulfilling $\rho_1(u)=\mathrm{tr}(w)=0$, and choose $v$ the diagonal unitary
with diagonal $(e^{\frac{2\pi i k}{n} })_{k=0}^{n-1}$.
\end{proof}

We can apply Theorem \ref{thm: the Avitzour approach} in other similar  situations. For instance, it follows from Theorem \ref{thm: the Avitzour approach} that $(A,\tau)*(\bM_{n}(\bC),\tr)$ is selfless if $n\geq 2$, $A$ is finite-dimensional, and $\tau$ is a trace with $\tau(z)<1/2$ for every central projection in $A$.

\section{Selflessness for more general free products}\label{sec: here is where the fun begins}

The main idea of the proof of Theorem \ref{thm: the Avitzour approach} can be adapted to show that $(A,\rho)*(C(\bT),m)$ is selfless, where $m$ is integration against Lebesgue measure and $(A,\rho)$ has a filtration with the rapid decay property. Indeed in this case, we simply adapt Lemma \ref{item: asy L2 L2 isometry avitzour} using the $*$-homomorphism $\phi_{n}\colon B_{1}*_{\textnormal{alg}}B_{2}*_{\textnormal{alg}}B_{3}\to A*C(\bT)$ (where $B_{1}=B_{3}=A,B_{2}=C(\bT)$), defined by declaring that $\phi_{n}$ is the identity on $B_{1},B_{2}$ and is conjugation by $u^{n}$ on $B_{3}$, where $u$ is the Haar unitary generating $C(\bT)$. The proof proceeds, with minor modifications, as that of Theorem \ref{item: asy L2 L2 isometry avitzour}. In this section, we expand on, and abstract, this strategy for free products with the circle to handle free products $(A_{1},\rho)*(A_{2},\tau)$ where $\tau$ is a trace.

\subsection{The general approach}

We give a general theorem in this section which will be our approach to deducing selflessness for other free products. 

Throughout this section we let $(A_1,\rho)$ and $(A_2,\tau)$
be unital $C^*$-probability spaces such that $\rho$ has faithful GNS representation and $\tau$ is a faithful trace.

\begin{lem}\label{lem: combinatorial strucutre of conjugation map}
Set $B_{1}=A_{1}=B_{3}$, $B_{2}=A_{2}$ 

For $v\in \cU(A_{2})$, define a *-homomorphism  $\phi_{v}\colon B_{1}*_{\textnormal{alg}}B_{2}*_{\textnormal{alg}}B_{3}\to A_{1}*A_{2}$ by declaring that $\phi_{v}$ is the identity on the first copy of $A_{1}$ and on $A_{2}$,  and that it is conjugation by $v$ on the second copy of $A_{1}$. 

Suppose that $a=a_{1}\cdots a_{\ell}\in B_{1}*_{\textnormal{alg}}B_{2}*_{\textnormal{alg}}B_{3}$ is an alternating centered word  in $(B_{1},B_{2},B_{3})$, with $\ell>0$. 

For $j\in [2]$, let $\Upsilon_{j}$ be the set of $x\in A_{j}$ such that there exists $1\leq i\leq \ell$ with $a_{i}\in (B_{s}\ominus \bC1)$ and $\iota_{s}(x)=a_{i}$, where $s\in [3]$ has the same parity as $j$ (i.e., $s\in \{1,3\}$ if $j=1$ and $s=2$ if $j=2$). Set 
\[\widetilde{\Upsilon}_{2}=(\Upsilon_{2}\cup \{1\})v\cup v^{*}(\Upsilon_{2}\cup \{1\})\cup v^{*}\Upsilon_{2}v \cup \Upsilon_{2}.\]

Then  
\[
\phi_{v}(a)=b_{1}\cdots b_{p}
\]
where:
\begin{itemize}
    \item $p$ depends only upon $a$ and not $v$, 
    \item $b_{i}\in A_{j(i)}$, where $j\colon [p]\to \{1,2\}$ is alternating and depends only upon $a$ (not upon $v$),
    \item $b_{i}\in \Upsilon_{1}$ if $j(i)=1$,
    \item $b_{i}\in \widetilde{\Upsilon}_{2}$ if $j(i)=2$.
    \item $p\leq 3\ell+2$.
\end{itemize}

\end{lem}

\begin{proof}
Write
\[
a=y_{0}\iota_{3}(c_{1})y_{1}\iota_{3}(c_{2})\cdots y_{t-1}\iota_{3}(c_{t})y_{t},
\]
where $c_{i}\in B_{3}\ominus \bC 1$, and $y_{i}$ is an alternating word in $(B_{1}\ominus \bC 1,B_{2}\ominus \bC1)$, with $y_{0},y_{t}$ potentially being empty. 
Then
\begin{align*}
\phi_{v}(a)&=y_{0}(vc_{1}v^{*})y_{1}(vc_{2}v^{*})\cdots (vc_{t}v^{*})y_{t}\\
&=(y_{0}v)c_{1}(v^{*}y_{1}v)c_{2}\cdots (v^{*}y_{t-1}v)c_{t}(v^{*}y_{t}).
\end{align*}
Note that $y_{0},y_{t}$ could be empty in which case $y_{0}v=v$ and $v^{*}y_{t}=v^{*}$. If $y_{0}$ is not empty, then $y_0v$ ends in $(\Upsilon_{2}\cup\{1\})v$ and its other letters are in $\Upsilon_{1}\cup \Upsilon_{2}$. Similarly, if $y_{t}$ is not empty, then $v^*y_t$ is an alternating word in $(\Upsilon_{1},\Upsilon_{2}\cup v^*(\Upsilon_{2}\cup\{1\}))$. 
If $1\leq i\leq t-1$ and $y_{i}$ is a single letter in $\Upsilon_{2}$, then $v^{*}y_{i}v$ is a single letter in $v^{*}\Upsilon_{2}v$. If $y_{i}$ is a single letter in $\Upsilon_{1}$, then $v^{*}y_{i}v$ is an alternating word in $(\Upsilon_{1},\{v,v^{*}\})$. Finally, if $y_{i}$ has length longer than $2$, then $v^{*}y_{i}v$ is an alternating word in 
\[
(\Upsilon_{1},v^{*}(\Upsilon_{2}\cup\{1\})\cup \Upsilon_{2}\cup (\Upsilon_{2}\cup \{1\})v),
\]
which begins in $v^{*}(\Upsilon_{2}\cup\{1\})$ and ends in $(\Upsilon_{2}\cup\{1\})v$. In all cases, we see that $v^{*}y_{i}v$ is an alternating word in $(\Upsilon_{1},\widetilde{\Upsilon}_{2})$ which begins and ends in $\widetilde{\Upsilon}_{2}$. Since each $c_{i}\in \Upsilon_{1}$, we see that $\phi_{v}(a)$ has the desired form. 

For the bound on $p$, let $\ell_{i}$ be the length of $y_{i}$, 
so that $t+\sum_{i}\ell_{i}=\ell$. Note that the length of $\phi_{v}(a)$ is increased by at most $2$ for every $y_{i}$. Thus the length of $\phi_{v}(a)$ is at most 
\[
t+2(t+1)+\sum_{i}\ell_{i}\leq 3\ell+2.\qedhere
\]
\end{proof}

We remark that nowhere in the above proof did we use that $\tau$ is a trace, and so the above lemma holds for more general free products. This may be of use in future investigations of selflessness for reduced free products. 


We now adapt the proof of Theorem \ref{thm: the Avitzour approach} for the case of free products with the circle (outlined at the start of this section) to more general free products. To do this we will need to adapt Lemma \ref{item: asy L2 L2 isometry avitzour}. We no longer assume that the  sequence of unitaries $v_{n}$ used to define $\phi_{n}$ belong to the filtration. Instead we use almost orthogonality conditions to guarantee that the maps are asymptotically trace-preserving and that the image of $\phi_{n}$ lands in a subspace where the operator norm is controlled by the $L^{2}$-norm.

\begin{lem}\label{lem: comb control in asymptotic case}

Let $(v_{k})_{k}$ be a sequence in $\cU(A_{2})$. Let $B_{1},B_{2},B_{3},$ be as in Lemma \ref{lem: combinatorial strucutre of conjugation map}. For each $k$, set $\phi_{k}=\phi_{v_{k}}$, where $\phi_{v_k}$ is a *-homomorphism defined as in Lemma \ref{lem: combinatorial strucutre of conjugation map}.

Let $a=a_{1}\cdots a_{\ell}$ be an alternating  word in $(B_{1}\ominus \bC 1)\cup (B_{2}\ominus \bC 1)\cup (B_{3}\ominus \bC 1)$ with $\ell>0$.
Let $\Upsilon_{1},\Upsilon_{2},\widetilde{\Upsilon}_{2}$ be defined as in Lemma \ref{lem: combinatorial strucutre of conjugation map}.
Suppose that 
\[\tau(v_{k}x)\to_{k\to\infty} 0 \textnormal{ for all $x\in \Upsilon_{2}\cup \{1\}\cup \Upsilon_{2}^{*}$.}\]
Finally, assume that $(\widehat{F}_{k})_k$ is a sequence of linear subspaces of $A_{2}$ such that 
\[
\lim_{k\to\infty}\max_{x\in \Upsilon_{2}}(\inf_{y\in \widehat{F}_{k}}\|x-y\|,\inf_{y\in \widehat{F}_{k}}\|v_{k}^{*}xv_{k}-y\|)=0.
\]
Set 
\[\Xi_{k}=(\id-\tau)((\Upsilon_{2}\cup\{1\})v_{k}\cup v_{k}^{*}(\Upsilon_{2}\cup\{1\})\cup \widehat{F}_{k}).\]
\begin{enumerate}[(i)]
\item 
There is a  $0<p\leq 3\ell+2$ depending only upon $a$, and a sequence $b_{k}$
of alternating words in $(\Upsilon_{1},\Xi_{k})$ of length $p$ such that 
\[
\|\phi_{k}(a)-b_{k}\|\to_{k\to\infty}0.
\]
\item we have $(\rho*\tau)(\phi_{k}(a))\to_{k\to\infty}0$. 
\end{enumerate}

\end{lem}

\begin{proof}
Set $\rho_{1}=\rho$ and $\rho_{2}=\tau$.     
By Lemma \ref{lem: combinatorial strucutre of conjugation map} we have that 
\[\phi_{k}(a)=\widetilde{b}_{1,k}\cdots \widetilde{b}_{p,k}\]
where:
\begin{itemize}
    \item $\widetilde{b}_{i,k}\in A_{j(i)}$ where $p$ and $j\colon [p]\to \{1,2\}$ is independent of $k$, 
    \item $\widetilde{b}_{i,k}\in  \Upsilon_{1}$ if $j(i)=1$,
    \item $p\leq 3\ell+2$,
    \item $\widetilde{b}_{i,k}\in (\Upsilon_{2}\cup\{1\})v_{k}\cup v_{k}^{*}(\Upsilon_{2}\cup \{1\})\cup v_{k}^{*}\Upsilon_{2}v_{k}\cup \Upsilon_{2}$, if $j(i)=2$. 
\end{itemize}
Set $\widehat{b}_{i,k}=\widetilde{b}_{i,k}-\rho_{j(i)}(\widetilde{b}_{i,k})$. Observe that $\|\widehat{b}_{i,k}\|\leq 2\|\widetilde{b}_{i,k}\|$.

Then:
\[\phi_{k}(a)=\widehat{b}_{1,k}\cdots \widehat{b}_{\ell,k}+\sum_{\varnothing\ne I\subseteq [\ell]}\left(\prod_{i\in I}\rho_{j(i)}(\widetilde{b}_{i,k})\right)\left(\prod_{i\in [\ell]\setminus I}\widehat{b}_{i,k}\right),\]
here the product $\prod_{i\in [\ell]\setminus I}\widehat{b}_{i,k}$ is interpreted as $\widehat{b}_{i_{1},k}\cdots \widehat{b}_{i_{s},k}$ if $[\ell]\setminus I=\{i_{1},\cdots, i_{s}\}$ with $i_{t}<i_{t+1}$, for $t=1,\cdots,s-1$. 
Since $\tau(v_{k}x)\to_{k\to\infty}0$ for all $x\in \Upsilon_{2}$, and $\tau$ is invariant under unitary conjugation, we see that \[\left(\prod_{i\in I}\rho_{j(i)}(\widetilde{b}_{i,k})\right)\to_{k\to\infty}0.\] By our control of the form of the $\widetilde{b}_{i,k}$ as well as the fact that $\|v_{k}\|=1$, we know that $\left\|\prod_{i\in [\ell]\setminus I}\widetilde{b}_{i,k}\right\|$ is bounded independent of $k$. Thus setting $\widehat{b}_{k}=\widehat{b}_{1,k}\cdots \widehat{b}_{p,k}$ proves that $\|\phi_{k}(a)-\widehat{b}_{k}\|\to_{k\to\infty}0$. The second item follows from the fact that free independence guarantees that $(\rho*\tau)(\widehat{b}_{k})=0$.

To finish the proof of the first item, note that by assumption we can find $b_{i,k}\in \widehat F_{k}$ such that $\|\widehat{b}_{i,k}-b_{i,k}\|\to_{k\to\infty}0$. Set $b_{k}=b_{1,k}\cdots b_{p,k}$ and observe that
\[\|b_{k}-\widehat{b}_{k}\|\leq \sum_{i=1}^{p}\left(\prod_{j=1}^{i-1}\|\widehat{b}_{j,k}\|\right)\|\widehat{b}_{i,k}-b_{i,k}\|\left(\prod_{j=i+1}\|b_{j,k}\|\right).\]
Since we previously saw that $\|\widehat{b}_{i,k}\|\leq 2\|\widetilde{b}_{i,k}\|$, the fact that $\|\widehat{b}_{i,k}-b_{i,k}\|\to_{k\to\infty}0$ and $\|\phi_{k}(a)-\widehat{b}_{k}\|\to_{k\to\infty}0$ establishes that 
\[\|\phi_{k}(a)-b_{k}\|\to_{k\to\infty}0.\]

\end{proof}

For later use, we upgrade the operator norm perturbation above to be given by a linear map close to $\phi_{k}$.

\begin{lem}\label{lem: asy L2-L2 isometry}
    Suppose we are given unital, linear subspaces $V_{j}\subseteq A_{j}$, for $j=1,2$.
Let $v_{k}\in\cU(A_{2})$ be a sequence of unitaries such that $\tau(v_{k}x)\to_{k\to\infty}0$ for all $x\in V_{2}$.
Let
$
\phi_{k}=\phi_{v_k}\colon B_{1}*_{\textnormal{alg}}B_{2}*_{\textnormal{alg}}B_{3}\to A_{1}*A_{2}
$
be given as in Lemmas \ref{lem: comb control in asymptotic case} and \ref{lem: combinatorial strucutre of conjugation map}.

 Suppose that $\widehat{F}_{k}\subseteq A_{2}$ is a sequence of linear subspaces with $1\in \widehat{F}_{k}$ such that 
    \[\lim_{k\to\infty}\max(\inf_{y\in \widehat{F}_{k}}\|y-x\|,\inf_{y\in \widehat{F}_{k}}\|y-v^{*}xv\|)=0,\]
    for all $x\in V_{n,2}$. Set 
    \[
    F_{k}=(\id-\tau)(V_{2}v_{k}+v_{k}^{*}V_{2}+\widehat{F}_{k}),
    \]
    and for $p\in \bN$, denote by $W_{p,k}\subseteq A_{1}*A_{2}$ the span of all alternating words in $(A_{1}\ominus \bC 1,F_{k})$ of length $p$. 
    Finally, for $n\in \bN$, denote by $\widehat{V}_{n}$  the span of all alternating words in $(\iota_{1}(V_{n,1}\ominus \bC 1), \iota_{2}(V_{n,2}\ominus \bC 1), \iota_{3}(V_{n,1}\ominus \bC 1))$ of length at most $n$. 


Then:
\begin{enumerate}[(i)]
    \item $(\rho*\tau)(\phi_{k}(a))\to_{k\to\infty} (\rho*\tau*\rho)(a)$ for all $a\in B_{1}*_{\textnormal{alg}}B_{2}*_{\textnormal{alg}}B_{3}$, \label{item: asymptotic trace-preserving}

\item 
\label{item: norm perturbation}there is a sequence of linear maps $\widetilde{\phi}_{k}\colon \widehat{V}_{n}\to \sum_{p=0}^{3n+2}W_{p,k}$ such that 
\[
\|\phi_{k}(a)-\widetilde{\phi}_{k}(a)\|\to_{k\to\infty}0 
\]
for all $a\in \widehat{V}_{n}$.
\end{enumerate}


\end{lem}

\begin{proof}
(\ref{item: asymptotic trace-preserving}): By linearity, it is enough to show that 
\[(\rho*\tau)(\phi_{k}(a_{1}\cdots a_{\ell}))\to_{k\to\infty}0\]
where $a_{1}\cdots a_{\ell}$ is an alternating word in $(B_{1}\ominus \bC 1, B_{2}\ominus \bC 1, B_{3}\ominus \bC 1)$ of positive length. By freeness, alternating centered words in $(A_{1}\ominus \bC 1, A_{2}\ominus \bC 1)$ have trace zero, and so the desired result follows from Lemma \ref{lem: comb control in asymptotic case}.


(\ref{item: norm perturbation}): 
For $\ell\in [n]\cup\{0\}$, let us denote by $\widehat{V}_{\ell,n}$ the span of all alternating words in $(\iota_{1}(V_{n,1}\ominus \bC 1),\iota_{2}( V_{n,2}\ominus \bC 1), \iota_{3}(V_{n,1}\ominus \bC 1))$ of length  exactly $\ell$.
More concretely, $\widehat{V}_{\ell,n}$ is the span of  the
words in $B_{1}*_{\textnormal{alg}}B_{2}*_{\textnormal{alg}}B_{3}$ of the form $a_{1}\cdots a_{\ell}$ where
    \begin{itemize}
        \item $\ell\leq n$,
        \item $a_{i}\in \iota_{j(i)}(V_{s(i)})\ominus \bC 1)$ $j\colon [\ell]\to [3]$ is alternating, $s\colon [\ell]\to [2]$ is such that $s(i)$ and $j(i)$ have the same parity,
        \end{itemize}

The spaces $(\widehat{V}_{\ell,n})_{\ell=0}^{n}$   are pairwise orthogonal, by freeness, and their sum is $\widehat{V}_n$. By the triangle inequality, it suffices to define $\widetilde{\phi}_{k}$ on each $\widehat{V}_{\ell,n}$ and show it has the desired norm perturbation property. For $\ell=0$, we simply take $\widetilde{\phi}_{k}$ to be the identity (identifying the copies of $\bC1$ inside $A_{1}*A_{2}$ and $A_{1}*A_{2}*A_{1}$). 

Fix a basis $(v_{j})_{j\in I_{\ell}}$ for $\widehat{V}_{\ell,n}$ where each 
\[v_{j}=a_{1,j}a_{2,j}\cdots a_{\ell,j}\]
and  $a_{1,j}a_{2,j}\cdots a_{\ell,j}$ is a word satisfying the above two bullet points. By Lemma \ref{lem: comb control in asymptotic case} we know that there is a fixed $0<p\leq 3n+2$ and $b_{k}^{(j)}\in W_{p,k}$ such that $\|\phi_{k}(v_{j})-b_{k}^{(j)}\|\to_{k\to\infty}0$. Let $\widetilde{\phi}_{k}$ be the unique linear map satisfying $\widetilde{\phi}_{k}(v_{j})=b_{k}^{(j)}$ for all $j\in I_{\ell}$. Applying the triangle inequality again, we see that $\|\phi_{k}(a)-\widetilde{\phi}_{k}(a)\|\to_{k\to\infty}0$ for all $a\in \widehat{V}_{\ell,n}$. 

\end{proof}

The above two lemmas will play the role of the $L^{2}$-$L^{2}$ isometry we exploited in the proof of Theorem \ref{thm: the Avitzour approach}.

We state our general approach precisely as follows.

\begin{thm}\label{thm: selfless from asy angle conditions}
Let $(A_{1},\rho),(A_{2},\tau)$ be unital $C^{*}$-probability spaces with $\tau$ tracial. Suppose that $A_{1}$ and $A_2$ have the rapid decay property
relative to filtrations $(V_{n,1})_{n=0}^{\infty}$ and $(V_{n,2})_{n=0}^{\infty}$.

Suppose that for each $n\in \bN$ there exist  a sequence of unitaries  $(u_{n,k})_k\in \cU(A_{2})$, and  linear subspaces $(\widehat{F}_{n,k})_k\subseteq A_{2}$, with $1\in \widehat{F}_{n,k}$, which are stable under $*$, and which satisfy:
\begin{itemize}
\item (almost orthogonality) 
\begin{align*}
\sup_{a_{1},a_{2}\in V_{n,2}:\|a_{1}\|_{2},\|a_{2}\|_{2}\leq 1}|\tau(a_{1}u_{n,k}a_{2}u_{n,k})|&\to_{k\to\infty}0,
\\
\sup_{a_{1}\in V_{n,2},a_{2}\in \widehat F_{n,k}:\|a_{1}\|_{2},\|a_{2}\|_{2}\leq 1 }|\tau(a_{1}u_{n,k}a_{2})|&\to_{k\to\infty} 0,
\end{align*}

\item (inflated rapid decay) there are constants $b,\beta>0$ (independent of $n$) such that 
\[
\|x\|\leq b(n+1)^{\beta}\|x\|_{2}
\]
for all $x\in \widehat{F}_{n,k}$ and all $k$.

\item (asymptotic containment), for all $y\in V_{n,2}\ominus \bC 1$ we have 
    \[\lim_{k\to\infty}\max(\inf_{c\in \widehat{F}_{n,k}}\|y-c\|,\inf_{c\in \widehat{F}_{n,k}}\|u_{n,k}^{*}yu_{n,k}-c\|)=0. \]
\end{itemize}
Then $A_{1}*A_{2}$ is selfless. 
\end{thm}

\begin{proof}[Proof of Theorem \ref{thm: selfless from asy angle conditions}]
For each $j=1,2$, choose $c_{j},\alpha_{j},j=1,2$ such that for all $n\in \bN$ and all   $\chi\in V_{n,j}$
\[
\|\chi\|\leq c_{j}(n+1)^{\alpha_{j}}\|\chi\|_{2}.
\]
Adopt notation as in Lemma \ref{lem: combinatorial strucutre of conjugation map} and consider  the *-homomorphisms 
\[
\phi_{n,k}=\phi_{u_{n,k}}\colon 
B_1*_{\textnormal{alg}}B_2*_{\textnormal{alg}}B_3\to A_1*A_2
\]
defined as in Lemma \ref{lem: combinatorial strucutre of conjugation map}. 

Consider the filtration $(V_n)_n$ of $A_{1}*A_{2}*A_{1}$ obtained from Lemma \ref{lem: defining the filtration for RDP} and the filtrations 
$(V_{n,1})_n$, 
$(V_{n,2})_n$, $(V_{n,1})_n$.

For $n,k\in \bN$, let 
\[
F_{n,k}=(\id-\tau)[V_{n,2}u_{n,k}+u_{n,k}^{*}V_{n,2}+\widehat{F}_{n,k}].
\]

For $p,n,k\in \bN$, let $W_{p,n,k}\subseteq A_1*A_2$ be the span of all words of length $p$ alternating in $(V_{n,1}\ominus \bC1,F_{n,k})$. That is, $W_{p,n,k}$ is the span of  words of the form $b_{1}\cdots b_{p}$, such that   
for an alternating $j\colon [p]\to \{1,2\}$ we have
$b_{i}\in V_{n,1}$ if $j(i)=1$, and 
$b_{i}\in F_{n,k}$ if $j(i)=2$.
By Lemma \ref{lem: asy L2-L2 isometry} (\ref{item: norm perturbation}) we may find linear maps $\widetilde{\phi}_{n,k}\colon V_{n}\to \sum_{p=0}^{3n+2}W_{p,n,k}$ satisfying
\begin{equation}\label{eqn: op norm perturbation}\|\phi_{n,k}(a)-\widetilde{\phi}_{n,k}(a)\|\to_{k\to\infty}0, \textnormal{ for all $a\in V_{n}$}.
\end{equation}
Lemma \ref{lem: asy L2-L2 isometry} (\ref{item: asymptotic trace-preserving}) also implies that
\begin{equation}\label{eqn: trace perturbation}(\rho*\tau)(\widetilde{\phi}_{n,k}(a))\to_{k\to\infty}(\rho*\tau*\rho)(a), \textnormal{ for all $a\in V_{n}$}.
\end{equation}

To prove the appropriate asymptotic norm estimate to apply Proposition \ref{wk convergence+ strong semi conv}, we turn to a rapid decay-type property for $W_{p,n,k}$. Choose $k_{0}(n)\in \bN$ such that for all $k\geq k_{0}(n)$ we have:
\begin{itemize}
    \item $|\tau(u_{n,k}a_{1}u_{n,k}a_{2})|\leq \frac16$ for all $a_{1},a_{2}\in V_{n,2}$ with $\|a_{1}\|_{2},\|a_{2}\|_2\leq 1$,
   
    \item $|\tau(a_{1}u_{n,k}a_{2})|\leq \frac{1}{6}$ for all  
    $a_{1}\in V_{n,2}$, $a_{2}\in \widehat{F}_{n,k}$ with $\|a_{1}\|_{2},\|a_{2}\|_2\leq 1$. 
\end{itemize}
Set $C=\max(c_{2},b)$ and $\gamma=\max(\alpha_{1},\beta)$. 

Let $\widehat{y}\in V_{n,2}u_{n,k}+u_{n,k}^{*}V_{n,2}+\widehat{F}_{n,k}$, and choose $x_{1},x_{2}\in V_{n,2}$
and $x_{3}\in \widehat{F}_{n,k}\ominus \bC 1$ so that $\widehat{y}=x_{1}u_{n,k}+u_{n,k}^{*}x_{2}+x_{3}$. Then
\[
\|\widehat{y}\|\leq \|x_{1}\|+\|x_{2}\|+\|x_{3}\|\leq C(n+1)^{\gamma}(\|x_{1}\|_{2}+\|x_{2}\|_{2}+\|x_{3}\|_{2}).
\]
Set $y_{1}=x_{1}u_{n,k}$, $y_{2}=u_{n,k}^{*}x_{2}$, $y_{3}=x_{3}$. Since $V_{n,2}$ and $\widehat{F}_{n,k}$ are stable under $*$, the above bullet points imply that for all $k\geq k_{0}(n)$: 
\begin{align*}
\|\widehat{y}\|_{2}^{2}=\sum_{j=1}^{3}\|x_{j}\|_{2}^{2}+\sum_{i,j\in [3]:i\ne j}\tau(y_{j}^{*}y_{i})&\geq \sum_{j=1}^{3}\|x_{j}\|_{2}^{2}-\frac{1}{6}\sum_{i,j\in [3]:i\ne j}\|x_{i}\|_{2}\|x_{j}\|_{2}\\
&\geq  \sum_{j=1}^{3}\|x_{j}\|_{2}^{2}-\frac{1}{6}\left(\sum_{j=1}^{3}\|x_{j}\|_{2}\right)^{2}\\
&\geq \frac{1}{2}\sum_{j=1}^{3}\|x_{j}\|_{2}^{2},
\end{align*}
the last line following by Cauchy-Schwarz. So we have shown 
\begin{equation}\label{eqn:pseudo proj}
\max_{i=1,2,3}\|x_{i}\|_{2}\leq 2\|\widehat{y}\|_{2}.
\end{equation}
 We see that:
\[|\tau(\widehat{y})|\leq \frac{1}{6}(\|x_{1}\|_{2}+\|x_{2}\|_{2})\leq \frac{2}{3}\|\widehat{y}\|_{2},\]
where in the last step we apply (\ref{eqn:pseudo proj}).
Thus:
\[\|(\id-\tau)(\widehat{y})\|_{2}^{2}=\|\widehat{y}\|_{2}^{2}-|\tau(\widehat{y})|^{2}\geq \frac{5}{9}\|\widehat{y}\|_{2}^{2}.\]
Combining with (\ref{eqn:pseudo proj}), it follows that 
\begin{align*}
\|(\id-\tau)(\widehat{y})\|\leq 2\|\widehat{y}\|
&\leq 2C(n+1)^{\gamma}\sum_{j=1}^{3}\|x_{j}\|_{2}\\
&\leq 12C(n+1)^{\gamma}\|\widehat{y}\|_{2}\\
&\leq \frac{36}{\sqrt{5}}C(n+1)^{\gamma}\|(\id-\tau)(\widehat{y}
)\|_{2}.
\end{align*}
In other words,
\begin{align*}\label{eqn: new inflated rapid decay}
\|y\|\leq   \frac{36}{\sqrt{5}}C(n+1)^{\gamma}\|y\|_{2} \textnormal{ for all $y\in F_{n,k}$ if $k\geq k_{0}(n)$.}  
\end{align*}
Set $B=2\max(c_{1},\frac{36}{\sqrt{5}}C)$ and $\sigma=\max(\alpha_{1},\gamma)$, it follows from Lemma \ref{lem: RDP estimate on layers for free prod} that 
\[\|x\|\leq B(p+1)(n+1)^{\sigma}\|x\|_{2} \textnormal{ for all $x\in W_{p,n,k},$ if $k\geq k_{0}(n)$.} \]
Let $\widetilde{x}\in \sum_{p=0}^{3n+2}W_{p,n,k}$ and write $\widetilde{x}=\sum_{p=0}^{3n+2}x_{p}$. Since the $x_{p}$ are pairwise orthogonal, we have 
\[\|\widetilde{x}\|\leq B(n+1)^{\sigma}\left(\sum_{p=0}^{3n+2}(p+1)\|x_{p}\|_{2}\right)\leq 3B(n+1)^{\sigma+3/2}\|\widetilde{x}\|_{2}, \textnormal{ if $k\geq k_{0}(n)$.}\]

Now fix $d\in V_{n}$, and $r\in \bN$. Then $(d^{*}d)^{r}\in V_{2nr}$. Thus
\[\limsup_{k\to\infty}\|\phi_{4nr,k}(d)\|^{2r}=\limsup_{k\to\infty}\|\phi_{4nr,k}((d^{*}d)^{r})\|=\limsup_{k\to\infty}\|\widetilde{\phi}_{4nr,k}((d^{*}d)^{r})\|,\]
where in the last step we use (\ref{eqn: op norm perturbation}).
For $k\geq k_{0}(4nr)$ we have 
\[\|\widetilde{\phi}_{4nr,k}((d^{*}d)^{r})\|\leq 3B(2nr+1)^{\sigma+3/2}\|\widetilde{\phi}_{4nr,k}((d^{*}d)^{r})\|_{2}.\]
Using (\ref{eqn: op norm perturbation}) again we have that 
\begin{align*}
\lim_{k\to\infty}\|\widetilde{\phi}_{4nr,k}((d^{*}d)^{r})\|_{2}=\lim_{k\to\infty}\|\phi_{4nr,k}((d^{*}d)^{r})\|_{2}&=\lim_{k\to\infty}(\rho*\tau)(\phi_{4nr,k}((d^{*}d)^{4r})^{1/2}\\
&=\|(d^{*}d)^{r}\|_{2}\leq \|d\|^{2r}.    
\end{align*}
with the second-to-last equality following by Lemma \ref{lem: asy L2-L2 isometry} (\ref{item: asymptotic trace-preserving}).
Altogether we have shown  
\[\limsup_{k\to\infty}\|\phi_{4nr,k}(d)\|\leq (3B)^{1/2r}(2nr+1)^{(\sigma+3/2)/2r}\|d\|.\]
So
\[\limsup_{r\to\infty}\limsup_{k\to\infty}\|\phi_{4nr,k}(d)\|\leq \|d\|, \textnormal{ for all $d\in V_{n}$.}\]
 Applying Proposition \ref{wk convergence+ strong semi conv} and a diagonal argument shows $A_{1}*A_{2}$ is selfless. 
\end{proof}

\section{Free products of separable $C^{*}$-algebras with rapid decay}\label{fun continues}

Recall that a tracial von Neumann algebra is a pair $(M,\tau)$ where $M$ is a von Neumann algebra and $\tau\in M^{*}$ is a faithful, normal, tracial state.

Recall that if $\omega$ is  free ultrafilter on $\bN$ and $(M_{n},\tau_{n})$ are tracial von Neumann algebras, then their \emph{tracial ultraproduct} is the von Neumann algebra
\[\frac{\{(x_{n})_{n}\in \prod_{n}M_{n}:\sup_{n}\|x_{n}\|<+\infty\}}{\{(x_{n})_{n}\in \prod_{n}M_{n}:\sup_{n}\|x_{n}\|<+\infty,\lim_{n\to\omega}\|x_{n}\|_{2}=0\}}.\]
If $(M,\tau)$ is a tracial von Neumann algebra, its \emph{central sequence algebra} is $M'\cap \cM$ where $\cM$ is the tracial ultrapower of $M$.

The goal of this section is to use  Theorem \ref{thm: selfless from asy angle conditions} to prove Theorem \ref{mainthm2} from the introduction, which we restate here:

\begin{thm}\label{thm:free prod RDP selfless}
Let $(A_{1},\rho),(A_{2},\tau)$ be unital $C^{*}$-probability spaces with $\tau$ tracial, 
$\dim_{\bC}(A_{j})>1$ for $j=1,2$, and $A_{2}$ separable. Suppose that $A_{1},A_{2}$ have filtrations with the rapid decay property.
Suppose that either:
\begin{itemize}
    \item the GNS completion of $A_{2}$ with respect to $\tau$ is a $\textrm{II}_{1}$-factor, or
    \item the GNS completion of $A_{2}$ with respect to $\tau$ has diffuse central sequence algebra (e.g. if the GNS completion of $A_{2}$ has diffuse center). 
\end{itemize}
Then $A_{1}*A_{2}$ is selfless.  
\end{thm}

 Let us first handle the case where the GNS completion of $A$ is a $\textrm{II}_{1}$-factor.

 \begin{prop}\label{prop: how to handle factorial case}
  Let $(A,\tau)$ be a unital, separable tracial $C^{*}$-probability space, and suppose that the GNS completion of $A$ with respect to $\tau$ is a $\textrm{II}_{1}$-factor. Then there is a sequence of unitaries $(u_{k})_{k=0}^{\infty}$ in $A$ such that 
  \begin{itemize}
      \item $\tau(u_{k}a)\to 0$ for all $a\in A$, 
      \item $\tau(u_{k}au_{k}b)\to 0$ for all $a,b\in A$, and
      \item $\tau(u_{k}^{*}xu_{k}y)\to 0$ for all $x,y\in A\ominus \bC 1$.
  \end{itemize}
 \end{prop}

 \begin{proof}
 Let $M$ be the GNS completion of $A$ with respect to $\tau$. We continue to use $\tau$ for the trace on $M$. Fix an ultrafilter $\omega\in \beta\bN\setminus \bN$ and let $(\cM,\tau_{\omega})$ be the tracial ultrapower of $(M,\tau)$.  By separability of $A$, we can apply \cite{popa1995free} to find a Haar unitary $u\in \cU(\cM)$ which is free from $M$. Write $u=e^{ih}$ with $h=h^{*}\in \cM$. By \cite[Theorem 3.3]{KRCentralsequence} we may lift $a$ to  a uniformly bounded sequence $(h_{k})_{k}\in A_{s.a.}$ Set $u_{k}=e^{ih_{k}}$, such that $u_{k}\in \cU(A)$. Since continuous functional calculus commutes with $*$-homomorphisms, we have that $(u_{k})_{k}$ is a lift of $u.$
Thus:
 \[\lim_{k\to\omega}\tau(u_{k}au_{k}b)=\tau_{\omega}(uaub)=0.\]
 Similarly,
 \[\lim_{k\to\omega}\tau(u_{k}^{*}xu_{k}y)=0, \textnormal{ and } \lim_{k\to\omega}\tau(u_{k}a)=0,\]
 for all $x,y\in M\ominus \bC 1$, $a\in M$.

Applying separability of $A$ and a diagonal argument now completes the proof.

 \end{proof}

\begin{prop}\label{prop: dealing with the diffuse center case}
 Let $(A,\tau)$ be a unital tracial $C^{*}$-probability space with $A$ separable.  Let $M$ be the GNS completion of $A$ with respect to $\tau$. For $\omega\in \beta\bN\setminus\bN$, let $\cM$ be the tracial ultrapower of $M$ with respect to $\omega$. Suppose that $M'\cap \cM$ is diffuse.

Then there is a sequence of unitaries $u_{k}\in \cU(A)$ such that 
\begin{itemize}
     \item $\|u_{k}x-xu_{k}\|\to_{k\to\infty}0$ for all $x\in A$,
     \item  $\tau(u_{k}x)\to 0$ for all $x\in A$,
     \item $\tau(u_{k}xu_{k}y)\to 0$ for all $x,y\in A$. 
 \end{itemize}
 
\end{prop}

\begin{proof}
Let $\cA$ be the $C^{*}$-ultrapower of $A$ with respect to $\omega$.
 By \cite[Theorem 3.3]{KRCentralsequence}, the natural $*$-homomorphism $q\colon A'\cap \cA\to M'\cap \cM$ is surjective. Since $M'\cap\cM$ is diffuse, we can find an injective, normal $*$-homomorphism $\Psi\colon L^{\infty}([-1/2,1/2])\to M'\cap \cM$ such that $\tau^{\omega}(\Psi(f))=\int_{-1/2}^{1/2}f(x)\,dx$ for all $f\in L^{\infty}([-1/2,1/2])$ (this is a folklore result, see e.g \cite[Chapter 3]{anantharaman-popa} or \cite[Proposition A.1]{LewisFrankMET}). 
Set $h=\Psi(t\mapsto 2\pi t)$. Since $q$ is surjective, we may choose a self-adjoint element $\widetilde{h}\in A'\cap \cA$ with $q(\widetilde{h})=h$. Replacing $\widetilde{h}$ with $f(\widetilde{h})$ where $f\colon \bR\to [-1/2,1/2]$ is continuous and the identity on $[-1/2,1/2]$, we may assume that $\|\widetilde{h}\|\leq 1/2$. Lift $\widetilde{h}$ to a sequence of self-adjoint elements $(\widetilde{h}_{k})_{k}$ in $A$. Set $u_{k}=e^{2\pi i \widetilde{h}_{k}}$. Since continuous functional calculus commutes with $*$-homomorphisms, we see that $(e^{2\pi in\widetilde{h}_{k}})_{k}$ is a lift of $e^{2\pi in \widetilde{h}}$ for every $n$ in $\bN$. Moreover $e^{2\pi in\widetilde{h}}\in A'\cap\cA$, since $A'\cap \cA$ is a unital $C^{*}$-algebra. By the Riemann-Lebesgue lemma and normality of $\Psi$, we have  $q(e^{2\pi in \widetilde{h}})=e^{2\pi i nh}\to_{n\to\infty}0$ in the weak operator topology acting on $L^{2}(\cM)$. Thus:
\begin{itemize}
    \item $\lim_{k\to\omega}\|xu_{k}^{n} -u_{k}^{n}x\|=\|xe^{2\pi i n\widetilde{h}}-e^{2\pi i n\widetilde{h}}x\|=0,$ \textnormal{ for all $n\in \bN$, and all $x\in A$},
    \item $\lim_{n\to\infty}\lim_{k\to\omega}\tau(u_{k}^{n}x)=\lim_{n\to\infty}\tau(e^{2\pi inh}x)=0$, \textnormal{ for all $x\in A$},
    \item $\lim_{n\to\infty}\lim_{k\to\omega}\tau(u_{k}^{n}xu_{k}^{n}y)=\lim_{n\to\infty}\tau(e^{4\pi i n h}xy)=0,$ for all $x,y\in A$. 
\end{itemize}
Separability of $A$ now allows us to apply a diagonal argument to complete the proof.

\end{proof}

We mention the following examples of von Neumann algebras which have diffuse central sequence algebra with respect to any normal trace.

\begin{prop}\label{prop: examples with diffuse central sequence}
Let $(M,\tau)$ be a tracial von Neumann algebra. Suppose that for every minimal nonzero central projection $z$ we have that $Mz$ is not full. Then $M$ has diffuse central sequence algebra.
\end{prop}

\begin{proof}
Fix $\omega\in\beta\bN\setminus \bN$. Splitting $Z(M)$ into diffuse and atomic pieces, we see that our hypothesis is equivalent to saying that there is a set $J$ (potentially empty) such that:
\[M=M_{0}\oplus \bigoplus_{j\in J}M_{j},\]
where:
\begin{itemize}
    \item $M_{0}$ is either $\{0\}$ or has diffuse center,
    \item $M_{j}$ is a nonfull $\textrm{II}_{1}$-factor.    
\end{itemize}
For $j\in \{0\}\sqcup J$, let $\cM_{j}$ be the von Neumann algebra ultrapower of $M_{j}$ with respect to $\omega$. Then:
\[M'\cap \cM\supseteq Z(M_{0})\oplus \bigoplus_{j\in J}M_{j}'\cap \cM_{j}:=N.\]
By  \cite[Proposition 1.10]{DixmierGamma} we know that $M_{j}'\cap \cM_{j}$ is diffuse for every $j\in J$. Since $Z(M_{0})$ is diffuse, it follows that $N$ is diffuse. 
Since $M'\cap \cM$ is a finite von Neumann algebra, it has a faithful, normal conditional expectation onto $N$.
 This implies that $M'\cap \cM$ is diffuse, (e.g. this follows from restricting such a conditional expectation to the maximal purely atomic direct summand of $M'\cap \cM$ and applying \cite[Theorem IV.2.2.3]{BlackadarOA}).
\end{proof}

\begin{rem}
If $M$ is diffuse, then we can write it as a direct sum of von Neumann algebras which either have diffuse centers or are $\textrm{II}_{1}$-factors. However, it is surprisingly subtle and unclear if we can pass the rapid decay assumptions to these direct summands. It is thus an interesting question for further research to adapt our arguments to the case where $M$ is diffuse.
\end{rem}

We now prove Theorem \ref{thm:free prod RDP selfless}.

\begin{proof}[Proof of Theorem \ref{thm:free prod RDP selfless}]

We first note that we may find a filtration $(V_{n})_{n=0}^{\infty}$ of $A_{2}$ such that $(V_{n})_{n=0}^{\infty}$ has the rapid decay property and $\dim(V_{n})<+\infty$ for every $n\in \bN$.

To prove this, fix a filtration $(E_{n})_{n=1}^{\infty}$ of $A_{2}$ with the rapid decay property. Since $A_{2}$ is separable, we may find an increasing sequence $(E_{n}^{(r)})_{r=1}^{\infty}$ of subspaces of $E_{n}$ such that $E_{n}^{(r)}$ is finite dimensional for every $r$, $1\in E_{n}^{(r)}$ for every $r$, and such that $\bigcup_{r}E_{n}^{(r)}$ is dense in $E_{n}$.
Further, set $E_{0}^{(r)}=\bC 1$ for all $r\in \bN$. 

Let
\[V_{n}=\sum_{\substack{j\in \{0,1,\cdots,n\}^{n}:j_{1}+\cdots+j_{n}\leq n\\
r\in \{1,\cdots,n\}^{n}}}\Span\{v_{1}\cdots v_{n}:v_{i}\in E_{j_{i}}^{(r_{i})}\}.\]
It is direct to check that this is a filtration. Each $V_{n}$ is finite dimensional, and since $\bigcup_{n}V_{n}\supseteq E_{k}^{(r)}$ for all $r,k$, we see that $\bigcup_{n}V_{n}$ is dense in $A_{2}$. Finally, since each $E_{n}$ is a filtration, we have that $V_{n}\subseteq E_{n}$ and thus the rapid decay property of $(V_{n})_{n=1}^{\infty}$ is inherited from $(E_{n})_{n=1}^{\infty}$.

We may thus fix a filtration
$(V_{n})_{n=0}^{\infty}$ of $A_{2}$ such that $(V_{n})_{n=0}^{\infty}$ has the rapid decay property and $\dim(V_{n})<+\infty$ for every $n\in \bN$. 

We now verify the asymptotically orthogonality, uniform nontrivial angle conditions, and inflated rapid decay hypotheses of Theorem \ref{thm: selfless from asy angle conditions}.

\emph{Case 1: The GNS completion of $A_{2}$ is a $\textrm{II}_{1}$-factor}.
Let $u_{k}$ be as in conclusion to Proposition \ref{prop: how to handle factorial case}. We verify the hypotheses of Theorem \ref{thm: selfless from asy angle conditions} with $u_{n,k}=u_{k}$, $\widehat{F}_{n,k}=u_{k}^{*}(V_{n}\ominus \bC 1)u_{k}+V_{n}\ominus \bC 1$.  The asymptotic containment condition is immediate in this case.
Since $\|u_{k}\|=1$, the functions
\begin{itemize}
    \item $(a,b)\in V_{n}\mapsto \tau(u_{k}ab)$,
    \item $(a,b)\in V_{n}^{2}\mapsto \tau(u_{k}au_{k}b)$,
    \item $(x,y)\in (V_{n}\ominus \bC 1)^{2}\mapsto \tau(u_{k}^{*}xu_{k}y)$,
\end{itemize}
are equicontinuous on the sets  $\{(a,b)\in V_{n}:\|a\|_{2},\|b\|_{2}\leq 1\}$, $\{(x,y)\in (V_{n}\ominus \bC 1)^{2}:\|x\|_{2},\|y\|_{2}\leq 1\}$  under the metric $d((c_{1},d_{1}),(c_{2},d_{2}))=\left(\|c_{1}-c_{2}\|_{2}+\|d_{1}-d_{2}\|_{2}\right)^{1/2}$. By finite-dimensionality, these sets are compact and so the above pointwise convergence can be upgraded to uniform convergence. Thus setting:
\begin{itemize}
    \item $\delta_{k,1}=\sup_{a,b\in V_{n}:\|a\|_{2},\|b\|_{2}\leq 1}|\tau(u_{k}ab)|$,
    \item $\delta_{k,2}=\sup_{(a,b)\in V_{n}^{2}:\|a\|_{2},\|b\|_{2}|
\leq 1}|\tau(u_{k}au_{k}b)|$,
    \item $\delta_{k,3}=\sup_{(x,y)\in (V_{n}\ominus \bC 1)^{2}:\|x\|_{2},\|y\|_{2}\leq 1} |\tau(u_{k}^{*}xu_{k}y)|$,
    \item $\delta_{k}=\max(\delta_{k,1},\delta_{k,2},\delta_{k,3})$,
\end{itemize}
we have $\delta_{k}\to 0$.  For the inflated rapid decay, let $x,y\in V_{n}\ominus \bC 1$. Then for all large $k$, we have 
\begin{align*}
\|u_{k}^{*}xu_{k}+y\|_{2}^{2}&\geq \|x\|_{2}^{2}+\|y\|_{2}^{2}-2\delta_{k}\|x\|_{2}\|y\|_{2}\\
&=(1-\delta_{k})(\|x\|_{2}+\|y\|_{2})^{2}+\delta_{k}(\|x\|_{2}-\|y\|_{2})^{2}\\
&\geq \frac{1}{2}\max(\|x\|_{2}^{2},\|y\|_{2}^{2}).    
\end{align*}
By assumption, we can choose $c,\alpha>0$ such that $\|a\|\leq c(n+1)^{\alpha}\|a\|_{2}$ for all $a\in V_{n}$.
Hence, for all large $k$:
\[\|u_{k}^{*}xu_{k}+y\|\leq \|x\|+\|y\|\leq C(n+1)^{\alpha}(\|x\|_{2}+\|y\|_{2})\leq 4C(n+1)^{\alpha}\|u_{k}^{*}xu_{k}+y\|_{2},\]
verifying the inflated rapid decay assumption.  For the asymptotic orthogonality, note that if $x,y\in V_{n}\ominus \bC 1$, and $a\in V_{n}$, then:
\begin{align*}
 |\tau(au_{k}(x+u_{k}^{*}yu_{k}))|=|\tau(u_{k}xa)+\tau(u_{k}ay)|&\leq \delta_{k}\|a\|_{2}(\|x\|_{2}+\|y\|_{2})\\
 &\leq 2\delta_{k}\|a\|_{2}\|u_{k}^{*}xu_{k}+y\|_{2}.   
\end{align*}
Since $\delta_{k}\to 0$, this verifies the asymptotic orthogonality condition.
We now apply Theorem \ref{thm: selfless from asy angle conditions} to complete the proof.

\emph{Case 2: The GNS completion of $(A_{2},\tau)$ has diffuse central sequence algebra.}
In this case, let $u_{k}$ be the unitaries from Proposition \ref{prop: dealing with the diffuse center case}, and set $u_{n,k}=u_{k}$. Fix $n\in \bN$. As in Case 1, we deduce the existence of positive numbers $\delta_{k}\to_{k\to\infty}0$ such that for all $a,b\in V_{n}$: 
\begin{itemize}
    \item $|\tau_{k}(u_{k}ab)|\leq \delta_{k}\|a\|_{2}\|b\|_{2}$
    \item $|\tau_{k}(u_{k}au_{k}b)|\leq \delta_{k}\|a\|_{2}\|b\|_{2}$,
    \item $\|u_{k}^{*}au_{k}-a\|\leq \delta_{k}\|a\|$.
    \end{itemize}
We set $\widehat{F}_{n,k}=V_{n}\ominus \bC 1$. The asymptotic containment condition is immediate from the last bullet point above. The inflated rapid decay of $\widehat{F}_{n,k}$ then follows from the assumed rapid decay of $V_{n}$, and the above bullet points directly imply the asymptotic orthogonality conditions. 
\end{proof}

We can now prove Theorem \ref{main semi thm} from the introduction.

\begin{thm}\label{theoremA}
   Voiculescu's free semicircular $C^*$-algebras $\mathcal S_n$ are selfless for $n\geq 2$. In particular, they have strict comparison. 
\end{thm}

\begin{proof}
We note  that
\[
\mathcal S_n\cong *_{i=1}^n (C[-2,2],\tau),
\]
where $\tau$ is given by integration against the semicircular distribution, i.e., $\tau(f)=\frac{1}{2\pi}\int_{-2}^{2}f(t)\sqrt{4-t^{2}}\,dt$. It will suffice to show that $\mathcal S_2$ is selfless,
since selflessness is preserved under taking reduced free product with an arbitrary $C^*$-probability $(A,\rho)$ (provided $\rho$ induces a faithful GNS representation) \cite[Theorem 4.2]{robert2023selfless}. Consider thus the reduced free product
\[
(C[-2,2],\tau)*(C[-2,2],\tau).
\]
By Example \ref{ex:sc}, both factors have rapid decay relative to the polynomial degree filtration. Since $\sqrt{4-t^{2}}\,dt$ is atomless, the GNS completion of $C[-2,2]$ relative to $\tau$ is isomorphic to $L^\infty([-2,2],m)$.
The second bullet point condition in Theorem \ref{thm:free prod RDP selfless} is thus met.  We conclude by this theorem that $\mathcal S_2$ is selfless, as desired.
\end{proof}

Let $H_\bR$ be a real Hilbert space and $(U_t)_t$ a one-parameter group
of orthogonal transformations on $H_\bR$. Let $\Gamma(H_\bR, U_t)$ denote
the free Araki-Wood $C^*$-algebra introduced by Shlyakhtenko in \cite{Shlyakhtenko}.
Let $\phi_U$ denote the vacuum state on $\Gamma(H_\bR, U_t)$. 
We refer the reader to \cite{Shlyakhtenko} for  details on this construction. We note that, in the case that $(U_t)_t$
acts trivially on $H_\bR$, $\Gamma(H_\bR, U_t)$ agrees with the $C^*$-algebra generated by a free semicircular
system. In the case that $(U_t)_t$ is nontrivial, Shlyakhtenko shows that $\Gamma(H_\bR, U_t)$ is simple and traceless. 

\begin{cor}
If $\dim H_\bR\geq 3$ and $(U_t)_t$ nontrivial, then $(\Gamma(H_\bR, U_t),\phi_U)$ is selfless, simple, and purely infinite. 
\end{cor}		

\begin{proof}
Since $\phi_U$ is faithful and non-tracial, simplicity and pure infiniteness are implied by selflessness \cite[Theorem 3.1]{robert2023selfless}, so we 
must show that $\Gamma(H_\bR, U_t)$ is selfless. 

We may split $H_{\bR}$ as a direct sum of 
$(U_t)_t$-invariant subspaces: $H_\bR = H_{c}\oplus H_{wm}$, where $H_{c}$ is a (potentially infinite) direct sum of finite-dimensional subrepresentations, and $H_{wm}$ has no nonzero finite-dimensional subrepresentations. Since $\Gamma(H_{\bR},U_{t})=\Gamma(H_{c},U_{t})*\Gamma(H_{wm},U_{t})$ (by \cite[Theorem 2.11]{Shlyakhtenko}), and the reduced free product of a selfless $C^{*}$-algebra with any $C^{*}$-algebra is selfless, it suffices to show that one of $\Gamma(H_{c},U_{t})$, $\Gamma(H_{wm},U_{t})$ is selfless. Since $\dim_{\bR}(H_{\bR})\geq 3$, we either have that $\dim_{\bR}(H_{wm})>0$ or $\dim_{\bR}(H_{c})\geq 3$.

Suppose that $\dim_{\bR}(H_{wm})>0$. Let $V_{t}$ be the complexification of $U_{t}|_{H_{wm}}$ acting on $\cK=H_{wm}\otimes_{\bR}\bC$. By Stone's theorem, we have that $V_{t}=\int_{\bR}e^{it \xi}dE(\xi)$ for some spectral measure $E$ defined on the Borel subsets of $\bR$. Let $J$ be the conjugation linear map $\id_{H_{wm}}\otimes_{\bR}C$ where $C$ is conjugation on $\bC$. Note that 
\[J\left(\int e^{it\xi}dE(t)\right)J=JV_{t}J=V_{t}=\int e^{it\xi}dE(t).\]
Fourier inversion thus implies that $J\left(\int f(t)dE(t)\right)J=\int \overline{f(-t)}\,dE(t)$ for all Schwartz function $f$, and thus for all bounded, Borel $f$ via a standard approximation. Note further that if $\widetilde{K}\subseteq \cK$ is a closed, linear subspace with $J\widetilde{K}=\widetilde{K}$, then $\widetilde{K}=\widetilde{H}\otimes_{\bR}\bC$ for a unique $\widetilde{H}\subseteq H$. If $\Omega$ is a symmetric subset of $\bR$, then by the above we have $JE(\Omega)J=E(\Omega)$, and so $\widetilde{K}=E(\Omega)(\cK)$  is the complexification of a $U_{t}$-invariant subspace.

Since $H_{wm}$ has no finite-dimensional subrepresentations, the above paragraph implies that $E$ has no atoms. Let $p_{t}=E((-\infty,t])$. Since $E$ has no atoms, $t\mapsto p_{t}$ is an SOT-continuous path of projections with $\lim_{t\to-\infty}p_{t}=0$ and $\lim_{t\to\infty}p_{t}=1$. Thus we may find an increasing sequence $0<t_{1}<t_{2}<\cdots$ of real numbers so that $t_{n}\to \infty$ and $p_{t_{i+1}}-p_{t_{i}}\ne 0$. Set $t_{0}=0$, and $K_{i}=(E((-t_{i},-t_{i-1}]\cup [t_{i-1},t_{i})))(K)$ for $i=1,2,\ldots$. By the above paragraph, there are  $U_{t}$-invariant subspaces $H_{i}$ of $H_{\bR}$ with $K_{i}$ the complexification of $H_{i}$ and each $H_{i}\ne 0$. Since $H_{wm}=\bigoplus_{i=1}^\infty H_{i}$, we have $\Gamma(H_{wm},U_{t})=*_{i=1}^\infty\Gamma(H_{i},U_{t})$. Since $H_{i}\ne 0$, each $\Gamma(H_{i},U_{t})$ has a state-preserving copy of $(C([-2,2]),\mu_{sc})$, where $\mu_{sc}$ is the semicircular measure. By the argument in Example \ref{ex: any diffuse measure on interval} we know that $(C([-2,2]),\mu_{sc})$ is isomorphic to $(C([0,1]),m)$ as $C^{*}$-probability spaces, where $m$ is Lebesgue measure. Thus $(C([-2,2]),\mu_{sc})$ has a state zero unitary. We now apply \cite[Theorem 2.8]{robert2023selfless} to see that $\Gamma(H_{wm},U_{t})$ is selfless.

Now suppose $\dim_{\bR}(H_{c})\geq 3$. If $\dim H_{c}=\infty$, then writing $H_{c}$ as a direct sum of finite-dimensional subrepresentations, we can write $\Gamma(H_{c},U_t)$ as a direct limit of C*-subalgebras
of the form $\Gamma(H_\bR',U'_t)$, with $3\leq \dim H_\bR'<\infty$. 
By the preservation of selflessness under direct limits \cite[Theorem 4.1]{robert2023selfless}, it suffices to show that each $\Gamma(H_{\bR}',U_t')$
 is selfless. We may thus assume that $H_{\bR}$ is finite dimensional. 

Suppose that $3\leq \dim H_\bR<\infty$. In this case $H_\bR$ decomposes into a direct sum
of 1-dimensional or 2-dimensional $U_t$-invariant subspaces. By \cite[Theorem 2.11]{Shlyakhtenko}, $\Gamma(H_\bR,U_t)$ can thus be
expressed as a reduced free product where each factor is either
$\Gamma(\bR, \mathrm{id}_t)$  or  $\Gamma(\bR^2,U_t)$, with $(U_t)_t$ non-trivial (see \cite[Remark 2.12]{Shlyakhtenko}). 
We note that $(\Gamma(\bR, \mathrm{id}_t),\tau_U)\cong (C[-2,2],\tau)$, with $\tau$ coming from the semicircular distribution. On the other hand, if $(U_t)_t$ acts nontrivially on $\bR^2$, then,
by \cite[Corollary 4.9]{Shlyakhtenko},
\[
(\Gamma(\bR^2,U_t),\phi_U)\cong (\mathcal T, \phi_\lambda)*(C[-2,2], \tau).
\] 
(Here $\mathcal T$ denotes the Toeplitz $C^*$-algebra and 
 $\phi_\lambda$ a certain faithful state on $\mathcal T$.) We deduce from these calculations  that if $\dim H_\bR\geq 3$, then
   \[
   (\Gamma(H_\bR,U_t),\phi_U)\cong (A,\rho)*(C[-2,2],\tau)*(C[-2,2],\tau),
   \]
for some $(A,\rho)$. It follows that $(\Gamma(\bR^2,U_t),\phi_U)$ is selfless by Theorem \ref{theoremA} and the preservation of selflessness under
reduced free products \cite[Theorem 4.2]{robert2023selfless}.
\end{proof}

 \section{Selflessness for free products of finite dimensional abelian $C^*$-algebras}

In this section we prove the following theorem:

\begin{thm}\label{abelianfindim}
Let	$A$ and $B$ be finite-dimensional abelian $\mathrm{C}^*$-algebras endowed with faithful states $\tau_A$ and $\tau_B$. Then
\[
(C,\tau)=(A,\tau_A)\ast(B,\tau_B),
\]
is selfless if and only if  $\dim(A)+\dim(B)\ge 5$ and  
whenever $p$ is a minimal projection in $A$, and $q$ is a minimal projection in $B$, we have
\begin{equation}\label{iffcondition}
	\tau_A(p) + \tau_B(q) < 1.
\end{equation}
\end{thm}	

We derive this theorem from Theorem~\ref{thm:free prod RDP selfless}  combined with Dykema's results on the structure of  reduced free products of finite dimensional abelian $C^*$-algebras, as well as on their von Neumann algebra counterpart.

\begin{lem}\label{sumsfreeprod}
	For $i=1,2,3$, let $(A_i,\tau_i)$ be  tracial $C^*$-probability spaces with rapid decay relative to some filtrations.
	Suppose that $(A,\tau)=(A_1,\tau_1)*(A_2\oplus A_3,\tau_2\oplus\tau_3)$ is simple, $A_1\neq \bC$,  and that either:
\begin{itemize}
    \item the GNS completion of $A_{2}$ with respect to $\tau$ is a $\textrm{II}_{1}$-factor, or
    \item the GNS completion of $A_{2}$ with respect to $\tau$ has diffuse central sequence algebra (e.g. if the GNS completion of $A_{2}$ has diffuse center). 
\end{itemize}
	Then $(A,\tau)$ is selfless. 
\end{lem}		

\begin{proof}
Let $p_2$ be the unit of $A_2$. Since $A$ is simple, $p_2$ is a full projection, 
and so $A$ is stably 
isomorphic to  $p_2Ap_2$. In the tracial case,   selflessness is preserved under stable isomorphism, by \cite[Theorems 4.3, 4.4]{robert2023selfless}. 
It thus suffices to show that $(p_2Ap_2,\frac{1}{\tau(p_2)}\tau)$ is selfless. By \cite[Proposition 2.8]{Dykemasimplicitystablerankonefree}, 
 \[
p_2Ap_2= p_2(A_1*(A_2\oplus A_3))p_2\cong p_2Bp_2*A_2,
 \]
 where 
 \[
 B=A_1*(A_2\oplus \bC).
 \]
Note that $p_2Bp_2\neq \bC$, as it contains $p_2A_1p_2\neq \bC$. Note also that $p_2Bp_2$
has rapid decay (relative to a suitable filtration),
by the preservation of this property through direct sums, reduced free products, and corners (Proposition \ref{RDpreservation} and 
Theorem \ref{thm: RDP for free products}).
We conclude by Theorem~\ref{thm:free prod RDP selfless} that $p_2Ap_2 \cong p_2Bp_2*A_2$ is selfless.
\end{proof}	

In the proof below we adopt the following notation from \cite{dykema2}:
\[
\stacked{A_1}{p_1}{\alpha_1}\oplus \cdots \stacked{A_n}{p_n}{\alpha_n}
\]
indicates that a $C^*$-probability space $(A_1\oplus \cdots \oplus A_n,\tau)$ satisfies that $\tau(p_i)=\alpha_i$ for all $i$, where $p_i$ denotes the unit of  $A_i$.

\begin{proof}[Proof of Theorem \ref{abelianfindim}]
Write
\begin{align*}
A=\stacked{\bC}{p_1}{\alpha_1}\oplus\cdots \oplus \stacked{\bC}{p_m}{\alpha_m},\\
B= \stacked{\bC}{q_1}{\beta_1}\oplus \cdots \oplus \stacked{\bC}{q_n}{\beta_n}.
\end{align*}		
The conditions \eqref{iffcondition} then read as $\alpha_i+\beta_j<1$ for all $i,j$.
By \cite[Theorem 1]{dykema2}, these conditions, together with $m+n\geq 5$, are necessary and sufficient for the reduced free product $C^*$-algebra  $C$ to be a simple $C^*$-algebra.
They are thus also necessary for $(C,\tau)$ to be selfless.
	
Let us prove sufficiency. Suppose that $m+n\geq 5$ and that $\alpha_i+\beta_j<1$ for all $i,j$. 	
Assume without loss of generality that $m\geq 3$ and $n\geq 2$. 
Let $p=p_1+p_2$. Since $C$ is a simple $C^*$-algebra (by \cite[Theorem 1]{dykema2}),  $pCp$ is stably isomorphic to $A$.
By the permanence of the selfless property under stable isomorphism in the tracial case (\cite[Theorems 4.3, 4.4]{robert2023selfless}), it suffices to show that $(pCp,\frac1{\tau(p)}\tau)$ is selfless.
By \cite[Proposition 2.8]{Dykemasimplicitystablerankonefree}, we have that
\[
pCp\cong
(\stacked{\bC}{p_1}{\alpha_1'}\oplus \stacked{\bC}{p_2}{\alpha_2'})*pB_1p,
\]
where $\alpha_1'=\alpha_1/(\alpha_1+\alpha_2)$, $\alpha_2'=\alpha_2/(\alpha_1+\alpha_2)$, and  
\[
B_1=\Big(\stacked{\bC}{p}{\alpha_1+\alpha_2}\oplus 
\stacked{\bC}{p_3}{\alpha_3}\oplus\cdots \oplus \stacked{\bC}{p_m}{\alpha_m}\Big)*
\Big(\stacked{\bC}{q_1}{\beta_1}\oplus \cdots \oplus \stacked{\bC}{q_n}{\beta_n}\Big).
\]

Let us consider first the case when $n+m>5$.
In this case, by Dykema's structure theorem on reduced free products of finite abelian $C^*$-algebras \cite[Theorem 1]{Dykemasimplicitystablerankonefree}, 
\[
B_1\cong A_0\oplus \bC^k
\] 
for some $k\geq 0$ and  an infinite dimensional $C^*$-algebra $A_0$ containing a closed two-sided ideal $A_{00}$ such that $A_{00}$ is simple
and $A_0/A_{00}$ is abelian and finite dimensional (possibly 0). What is more relevant for our current argument is that, by Dykema’s analogous structure result for von Neumann algebra free products \cite[Theorem 2.3]{dykemafreehyper},
the GNS completion of $A_0$ in $B(L^2(A_0,\tau_{A_0}))$ is an interpolated free group factor $L(F_s)$. We then have that
 \[
pB_1p=(pA_0p)\oplus \bC^{k'}.
 \] 
Note that $(pA_0p)''\cong pA_0''p$ is still an interpolated free group factor. Note also that
$B_1$ has rapid decay, by Theorem \ref{thm: RDP for free products}, and so $pA_0p$ has rapid decay, by Proposition \ref{RDpreservation}. We conclude  by  Lemma
\ref{sumsfreeprod} that
$pCp\cong \bC^2*((pA_0p)\oplus \bC^{k'})$ is selfless, as desired.
 
Suppose now that $m+n=5$, i.e., $m=3$ and $n=2$. In this case,
\[
B_1=\Big(\stacked{\bC}{p}{\alpha_1+\alpha_2}\oplus 
\stacked{\bC}{p_3}{\alpha_3}\Big)*
\Big(\stacked{\bC}{q_1}{\beta_1}\oplus \stacked{\bC}{q_2}{\beta_2}\Big).
\]
If  $\alpha_1+\alpha_2+\beta_j=1$ for some $j$, then  $\alpha_3=\beta_j$. But this implies that  $\alpha_3+(1-\beta_j)=1$, which violates
\eqref{iffcondition}. Thus, $\alpha_{1}+\alpha_{2}+\beta_j\neq 1$
for $j=0,1$. It follows by the structure theorem for reduced free products of the form $\bC^2*\bC^2$ \cite[Proposition 2.7]{Dykemasimplicitystablerankonefree},
that
\[
B_1\cong M_2(C[a,b])\oplus \bC^k,
\]
where $k\leq 2$, and where the trace $\tau_{B_1}$ on the summand $M_2(C[a,b])$ is given by an atomless
measure with support $[a,b]$. We note that in this isomorphism the component of the projection $p$ in the summand $M_2(C[a,b])$ is a projection of constant rank one. We deduce that
 \[
pB_1p\cong C[a,b]\oplus \bC^{k'}.
 \]
Note that the GNS completion of $C[a,b]$ relative to  the  trace $\frac{1}{\tau_{B_1}(p)}\tau_{B_1}$   is 
a diffuse abelian von Neumann algebra. Note also 
that $C[a,b]$ has rapid decay relative to a suitable filtration, by Example \ref{ex: any diffuse measure on interval}.
Thus, as in the previous case, we conclude that $pCp$ is selfless through an application of Lemma \ref{sumsfreeprod}. 
\end{proof}

\end{document}